\pdfoutput=1
\documentclass[english]{interact} 
\usepackage[utf8]{inputenc}
\usepackage{mathtools}
\usepackage{bbm}
\usepackage{comment}
\usepackage{hyperref}
\usepackage{xcolor}
\usepackage[sort, numbers, square, compress]{natbib}

\usepackage[capitalise,nameinlink]{cleveref}
\numberwithin{equation}{section}
\usepackage{enumitem}

\newtheorem{theorem}{Theorem}[section]

\newtheorem{lemma}{Lemma}[section]
\newtheorem{proposition}{Proposition}[section]

\newtheorem{definition}{Definition}[section]

\usepackage{enumitem}
\newlist{assumption}{enumerate}{1}					
\setlist[assumption]{label=(\textsc{a}\arabic*)}
\crefname{assumptioni}{Assumption}{Assumptions}

\newlist{assumption2}{enumerate}{1}
\setlist[assumption2]{label=(\textsc{a}2)}			
\crefname{assumptioni}{Assumption}{Assumptions}

\newlist{assumption3}{enumerate}{1}
\setlist[assumption3]{label=(\textsc{a}6')}			
\crefname{assumptioni}{Assumption}{Assumptions}

\newcommand{\R}{\mathbb{R}}

\newcommand{\dist}{\mathrm{dist}}
\newcommand{\norm}[1]{\left\|#1\right\|}

\newcommand{\supp}{\mathrm{supp}}

\DeclareMathOperator{\diam}{\mathrm{diam}}

\usepackage[textsize=footnotesize]{todonotes}

\begin{document}
	
\title{Generalized Differentiability and Second-Order Necessary Optimality Conditions for an Elliptic Optimal Control Problem with Exponential Nonlinearity and Discrete Measures}
\author{
	\name{Vu Huu Nhu\textsuperscript{a}\thanks{CONTACT Vu Huu Nhu. Email: nhu.vuhuu@phenikaa-uni.edu.vn}, Nguyen Hai Son\textsuperscript{b}, Phan Quang Sang\textsuperscript{a}, \and Tran Duy\textsuperscript{a}}
	\affil{\textsuperscript{a} Faculty of Fundamental Sciences, PHENIKAA University, Duong Noi, Hanoi 12116, Vietnam}
	\affil{\textsuperscript{b} Faculty of Mathematics and Informatics, Hanoi University of Science and Technology, Hanoi, Vietnam}
}

\maketitle

\begin{abstract}
	This paper deals with  generalized differentiability and second-order necessary optimality conditions for a box-constrained optimal control problem governed by an exponential semilinear elliptic equation with discrete measures as sources, where the control belongs to the space of absolutely summable  sequences.
	The presence of the exponential nonlinearity and discrete measures makes the analysis particularly challenging. 
	In particular, the control-to-state operator may fail to be directionally differentiable.
	To address this issue, we first establish  finite-dimensional directional differentiability of the control-to-state operator; that is, the operator is directionally differentiable along directions contained in finite-dimensional subspaces of the control space. 
	We then introduce a notion of generalized derivative defined as the limit of the associated finite-dimensional directional derivatives as the dimension of these subspaces tends to infinity. 
	Based on this concept, together with estimates for first- and second-order Taylor-type expansions of  
	the exponential Nemytskii  operator associated with the control-to-state mapping, we derive first- and second-order generalized differentiability of the reduced objective functional. 
	This leads to first- and second-order necessary optimality conditions for the optimal control problem.
\end{abstract}


\begin{amscode}
	49K20, 49J20,  35J25, 35J61, 35J75
\end{amscode}

\section{Introduction}

In this paper, we consider the following optimal control problem
\begin{equation}
	\label{eq:P}
	\tag{P}
	\left\{
	\begin{aligned}
		& \min\limits_{\bm{u} = (u_i) \in \ell^1} J(\bm{u}) := \frac{1 }{2} \norm{y_{\bm{u}} - y_d}_{L^2(\Omega)}^2 + \frac{\nu }{2} \sum_{i=1}^\infty u_i^2 \\
		& \text{subject to } \bm{\alpha} \leq \bm{u} \leq \bm{\beta},
	\end{aligned}
	\right.
\end{equation}
where $\nu \geq 0$ and $y_{\bm{u}}$ is the weak solution to
the  Dirichlet problem
\begin{equation}
	\label{eq:P-state}
	\left\{
	\begin{aligned}
		-\Delta y + (e^y - 1) &= f_0(x) + \sum_{i =1}^{\infty} u_i \delta_{x_i} && \text{in } \Omega, \\
		y &= 0 && \text{on } \Gamma
	\end{aligned}
	\right.
\end{equation}
for $f_0 \in L^p(\Omega)$ with $p>1$ and $u_i \in \R$, $i \geq 1$, being $i$-th component of $\bm{u}$.
We assume that $\bm{\alpha}, \bm{\beta} \in \ell^1$ with $\alpha_i \leq \beta_i < 4\pi$ for all $i \geq 1$, $y_d \in L^2(\Omega)$,  $\Omega \subset \R^2$ is a bounded domain with Lipschitz boundary $\Gamma$, and $\{x_1, x_2, \ldots\}$ is a set of distinct fixed points in $\Omega$.	Here, 
$\ell^1$ denotes the space of all absolutely summable sequences. For given $\bm{v} = (v_i), \bm{w} = (w_i) \in \ell^1$, the symbol $\bm{v} \leq \bm{w}$ means that $v_i \leq w_i$ for all $i \geq 1$.

\medskip

Partial differential equations (PDEs)  with exponential nonlinearities and singular data, particularly those modeled by Dirac delta distributions, arise at the intersection of analysis, geometry, and mathematical physics. 
A prominent example is the Liouville-type equation, which appears naturally in problems involving the prescription of Gaussian curvature, conformal metrics, and mean field models; see, for instance, \cite{ChenLin2015,BandleFlucher1996,Ni1982}.


PDE-constrained optimal control problems involving measures have attracted considerable attention from the mathematical community; see, e.g., \cite{CasasClasonKunisch2012,CasasKunisch2014,ClasonKunisch2012,Hoppe2023,Otarola2024,CasasClasonKunisch2013,CasasZuazua2013}.

For problems governed by linear PDEs with convex objective functionals, we refer to \cite{ClasonKunisch2012,CasasClasonKunisch2012,CasasClasonKunisch2013,CasasZuazua2013}. In these works, with the exception of \cite{CasasZuazua2013}, the first-order optimality conditions and the numerical approximations are investigated, while second-order (both necessary and sufficient) optimality conditions are not required in this setting due to the convexity of mentioned optimization problems. 
Interestingly, although the controls belong to spaces of measures,  Casas et al. \cite{CasasClasonKunisch2012,CasasClasonKunisch2013} showed that optimal controls can be effectively approximated by linear combinations of Dirac measures, which is particularly relevant for practical applications as it allows one to control of the PDEs via finitely many point sources.
The work \cite{CasasZuazua2013}, on the other hand, establishes the first-order optimality system for the problems with both control constraint and state constraint.

Regarding problems subject to nonlinear elliptic PDEs, 
relevant contributions include those of Casas \& Kunisch \cite{CasasKunisch2014}  and Otárola \cite{Otarola2024} for semilinear elliptic equations,  
as well as Hoppe \cite{Hoppe2023} for quasilinear elliptic equations.
In these works, the control-to-state operator, which maps each control to the associated unique state of the PDE, is of class $C^2$;
consequently, the corresponding reduced cost functional is also twice continuously differentiable.
Based on this regularity, both first- and second-order optimality conditions are established.
In addition, stability properties and numerical error estimates for approximations  are, respectively, studied in \cite{CasasKunisch2014} and \cite{Otarola2024}.
A key assumption in these works is that the nonlinearity in the state equation satisfies a polynomial growth condition.

More recently, in \cite{Nhu2025}, an optimal control problem without box constraints on the control was investigated, 
in which the state equation takes the form \eqref{eq:P-state}, except that the source term is given by a \emph{finite} linear combination of Dirac measures concentrated at a fixed set of distinct points. 
In that setting, the state equation is generally ill-posed in the sense that it may admit no solution, due to the interplay between the exponential nonlinearity and the singular measure data. 
To derive first-order optimality conditions, a regularization approach combined with a limiting procedure was employed, in the spirit of \cite{Barbu1984,Tiba1990,MeyerSusu2017,Constantin2017,NeittaanmakiTiba1994}.

\medskip 

The main aim of this paper is to establish first- and second-order optimality conditions for \eqref{eq:P}. 
In marked contrast to the smooth settings in the research works listed above, with the exception of \cite{Nhu2025}, 
the associated control-to-state operator may fail to be directionally differentiable. This makes the analysis challenging. 
Unlike the approach adopted in \cite{Nhu2025}, we first show in \cref{thm:directional-diff-control2state} that the control-to-state operator $S$ is directionally differentiable  along \emph{truncated} directions belonging to finite-dimensional subspaces $\ell^1_k$ (see \eqref{eq:truncated-subspace}) of $\ell^1$
at  controls $u \in \ell^1$ such that $\exp(S(u)) \in L^r(\Omega)$  for some $r >1$--this regularity of $\exp(S)$ is guaranteed by a suitable upper-bound assumption; see \eqref{eq:beta-rho-ass} below.
By letting $k \to \infty$, we define a notion of \emph{generalized differentiability} of $S$; see \cref{def:generalized-der-S} and \cref{prop:generalized-dir-S} below. 
The chain rule, together with  some Taylor-type expansions of the exponential Nemytskii superposition operator $\exp(S)$, allows us to derive the finite-dimensional directional differentiability and generalized differentiability of the associated objective functional. 
These notions are then instrumental in formulating first- and second-order necessary optimality conditions; see \eqref{eq:1st-NOC-generalized} and \cref{thm:2ndNOCs}.

\medskip 
The paper is organized as follows. 
The notation and main assumptions for \eqref{eq:P} are introduced in \cref{sec:assumption}.
In \cref{sec:elliptic-equation-measure}, we derive exponential estimates for solutions to linear elliptic equations with discrete measures as sources. 
These estimates are applied in \cref{sec:existence-minimizers} to establish the $L^r$-regularity of $\exp(S)$ for some constant $r >1$.
\cref{sec:differentiability} is devoted to the study of finite-dimensional directional differentiability and generalized differentiability of the control-to-state mapping and the reduced cost functional. 
These notions are subsequently used to formulate the first- and second-order necessary optimality conditions, which are presented in \cref{sec:NOCs}. 
In addition, \cref{sec:optimal-control-prob}  addresses the existence of minimizers for \eqref{eq:P}, as well as provides characterizations of the tangent cone to the admissible set and a critial cone associated with \eqref{eq:P}. 
The paper concludes with appendices presenting auxiliary estimates for a real-valued exponential function  and for Poisson's equation with regularized point sources.


\section{Notation and main assumptions} \label{sec:assumption}

\paragraph*{Notation.}
The notation $B_X(x, \rho)$ stands for the open ball in a Banach space $X$, centered at $x$ with radius $\rho >0$. 
For a real number $s$, we use the symbols $s^+ := \max\{s,0 \}$ and $s^{-} := \max\{-s, 0\}$.
In $\ell^1$, with a slight abuse of notation, the symbol for a null vector is $\bm{0}$. For a given $\bm{\omega} = (\omega_i) \in \ell^1$, we  write
\[
	\bm\omega_{\max} := \max\{ \omega_i \mid i \geq 1 \}. 
\]
The usual norm in $\ell^1$ is denoted by $\norm{\cdot}_{\ell^1}$, defined by
\[
	\norm{\bm\omega}_{\ell^1} := \sum_{i=1}^\infty |\omega_i| \quad \text{for all } \bm{\omega} = (\omega_i) \in \ell^1.
\]
Now, for each integer $k \geq 1$, we denote by $\ell^1_k$ the finite-dimensional subspace of $\ell^1$, defined by
\begin{equation} \label{eq:truncated-subspace}
	\ell^1_k := \left\{ \bm{h} = (h_i) \in \ell^1 \middle | h_i = 0 \ \text{for all } i \geq k +1 \right\}.
\end{equation} 
For a given $\bm{h} = (h_i) \in \ell^1$ and an integer $k \ge 1$, we denote by
$\bm{h}_k := (h_i^k)$ the truncation of $\bm{h}$ to the first $k$ components,
defined by
\[
	h^k_i =
\left\{
\begin{aligned}
	& h_i && \text{if } 1 \leq i \leq k,\\
	& 0 && \text{if } i \geq k+1.
\end{aligned}
\right.
\]

\medskip

\paragraph*{Assumptions:}
\begin{assumption} \label{ass:standing}
	\item \label{ass:domain}
	$\Omega$ is a bounded domain in $\R^2$ with a Lipschitz boundary $\Gamma$. $\{x_1, x_2, \ldots\}$ is a set of fixed and pairwise disjoint points in the domain $\Omega$, which does not contain any of its accumulation points.
	Assume further that $y_d \in L^2(\Omega)$ and $f_0 \in L^p(\Omega)$ for some $p >1$.
	
	\item \label{ass:bounds-admissible-set}
	There holds  $\bm{\alpha}, \bm{\beta} \in \ell^1$ satisfying $\alpha_i \leq \beta_i < 4\pi$ for all $i \geq 1$.
	Moreover, the following condition
	\begin{equation}
		\label{eq:beta-rho-ass}
		\sum_{i =1}^\infty \beta_i^+ \ln \frac{1}{\rho_i} < \infty
	\end{equation}
	is satisfied,
	where $\rho_i$, $i \geq 1$, is defined by \eqref{eq:rho-i-def} below.
\end{assumption}
Throughout the paper, we impose \ref{ass:domain}. Assumption  \ref{ass:bounds-admissible-set} shall be used only in \cref{sec:NOCs} to derive the $L^r$-regularity, $r >1$, of $\exp(S)$--the exponential Nemytskii superposition operator associated with the control-to-state mapping $S$, defined in \cref{sec:control2state-mapping}.

\medskip 

From now on, we define, for each integer $i \geq 1$, the number
\begin{equation}
	\label{eq:rho-i-def}
	\rho_i := \min\left\{ \frac{1}{2} \inf \left\{ \dist(x_i,x_j) : j \geq 1, j \neq i \right\}, \dist(x_i, \Gamma) \right\}.
\end{equation}

The positivity of $\rho_i$, $i \geq 1$, is shown in \cref{prop:rho-i} below. 
\begin{proposition}
	\label{prop:rho-i}
	Under \Cref{ass:domain}, there hold $\rho_i >0$ for all $i \geq 1$ and 
	\begin{equation}
		\label{eq:square-rho-i-sum}
		\sum_{i=1}^\infty \rho_i^2 \leq R^2
	\end{equation}
	with $R := \frac{1}{2} \diam(\Omega)$. 
\end{proposition}
\begin{proof}
	The argument showing the positivity of all $\rho_i$ is standard. Moreover, by definition of $\rho_i$, the open balls $B_{\R^2}(x_i, \rho_i)$ are  pairwise disjoint and contained in $\Omega$. This proves \eqref{eq:square-rho-i-sum}.
\end{proof}

\section{Exponential estimates for elliptic equations with atomic measures}
\label{sec:elliptic-equation-measure}

We begin this section by deriving exponential estimates for solutions to Poisson's equation with an atomic measure as the source. We first introduce the following notation.
For a given $\bm{\omega} \in \ell^1$, we define
\begin{equation}
	\label{eq:L-func}
	L(\bm{\omega}, \bm{\rho}) := \sum_{i =1}^\infty \omega_i^{+} \ln \frac{1}{\rho_i},
\end{equation}
where $\rho_i$, $i \geq 1$, is defined in \eqref{eq:rho-i-def}.

The following proposition generalizes the one in \cite[Prop.~3.1]{Nhu2025} (see also  \cite[Thm.~1]{BrezisMerle1991}).
\begin{proposition}
	\label{prop:Poisson-exponential-esti-inf-many-Dirac}
	Assume that $\bm{\omega} \in \ell^1$ such that $\omega_i >0$ for all $i \geq 1$ and that 
	\begin{equation}
		\label{eq:omega-rho-ass}
		L(\bm{\omega}, \bm{\rho})  < \infty,
	\end{equation}
	where $L(\bm{\omega}, \bm{\rho})$ is defined as in \eqref{eq:L-func}.
	Let $y \in W^{1,q}_0(\Omega)$, $1 \leq q <2$, be a unique solution to 
	\begin{equation}
		\label{eq:Poisson-many-Dirac}
		-\Delta y  = \sum_{i=1}^\infty \omega_i \delta_{x_i} \, \text{in } \Omega, \quad
		y = 0 \, \text{on } \Gamma.
		%
	\end{equation}
	Then, for any $\alpha \in (0, 4 \pi)$, there holds
	\begin{multline}
		\label{eq:exponential-esti-Dirac-inf}
		\int_\Omega \exp \left[\frac{(4\pi - \alpha)|y(x)|}{\bm\omega_{\max}} \right] dx \\
		 \leq \frac{4\pi^2 R^2}{\alpha} (2R)^{\left(2 - \frac{\alpha}{2\pi}\right) \frac{\norm{\bm{\omega}}_{\ell^1}}{\bm{\omega}_{\max}} }
		\exp\left[\frac{1}{\bm\omega_{\max}}\left(2 - \frac{\alpha}{2 \pi}\right)   L(\bm{\omega}, \bm{\rho})\right]
	\end{multline}
	 with $R := \frac{1}{2} \diam \Omega$.
\end{proposition}
\begin{proof}
Without loss of generality, assume that $\Omega \subset B_{\R^2}(0,R) =: B_R$.
Let $k \geq 1$ be an arbitrary integer and let $y_k$ be the unique solution in $W^{1,q}_0(\Omega)$, $1 \leq q < 2$, to
\begin{equation}
	\label{eq:Poisson-many-Dirac-k}
	-\Delta y_k  = \sum_{i=1}^k \omega_i \delta_{x_i} \, \text{in } \Omega, \quad
	y_k = 0 \, \text{on } \Gamma.
	%
\end{equation}
By \cite[Thm.~4.B.1]{BresisMarcusPonce2007}, there holds $y_k \to y$ in $W^{1,q}_0(\Omega)$ for all $1 \leq q < 2 = \frac{N}{N-1}$. There then exists a subsequence, denoted in the same way, of $\{k\}$ such that 
\begin{equation}
	\label{eq:yk-y-limit}
	y_k \to y \quad \text{a.e. in } \Omega \quad \text{as } k \to \infty.
\end{equation}
We will show that
\begin{multline}
	\label{eq:exponential-esti-Dirac-k}
	\int_\Omega \exp \left[\frac{(4\pi - \alpha)|y_k(x)|}{\bm\omega_{\max}} \right] dx \\
	 \leq 
	\pi (2R)^{\left(2 - \frac{\alpha}{2\pi}\right) \frac{\sum_{i=1}^k \omega_i}{\bm{\omega}_{\max}} }
	\left[R^2 +\left(\frac{4\pi}{\alpha} -1\right) \sum_{i=1}^{k} \rho_i^2\right]
	\exp\left[\frac{1}{\bm\omega_{\max}}\left(2 - \frac{\alpha}{2 \pi}\right)   \sum_{i=1}^k \omega_i \ln \frac{1}{\rho_i} \right].
\end{multline}
From this, by letting $k \to \infty$ and using the limit \eqref{eq:yk-y-limit}, we derive \eqref{eq:exponential-esti-Dirac-inf} 
from Fatou's lemma, the estimate \eqref{eq:square-rho-i-sum}, and the condition \eqref{eq:omega-rho-ass}.

Let $\epsilon_0 \in (0, \frac{1}{2}\min\{\rho_i \mid 1 \leq i \leq k \})$
be arbitrary but fixed, and let $\{\phi_\epsilon\}_{0 < \epsilon < \epsilon_0}$ be a family of mollifiers. 
Assume that $y_i^{\epsilon}$ and $\tilde{y}_i^{\epsilon}$, $1 \leq i \leq k$, uniquely solve 
\begin{align}
	& 	\label{eq:Poisson-delta-i-G-BR}
	\left\{
	\begin{aligned}
		-\Delta y_i^{\epsilon} & = \phi_\epsilon(\cdot -x_i) && \text{in } \Omega, \\
		y_i^{\epsilon} &= 0 && \text{on } \Gamma
	\end{aligned} 
	\right.
	\quad \text{and}
	\quad
	\left\{
	\begin{aligned}
		-\Delta \tilde y_i^{\epsilon} & = \phi_\epsilon(\cdot -x_i) && \text{in } B_R, \\
		\tilde y_i^{\epsilon} &= 0 && \text{on } \partial B_R,
	\end{aligned}
	\right.
\end{align}
respectively. 
By using the maximum principle; see, e.g. \cite[Prop.~4.B.1]{BresisMarcusPonce2007}, we deduce that $y_i^{\epsilon} \geq 0$ a.e. in $\Omega$,  $\tilde y_i^{\epsilon} \geq 0$ a.e. in $B_R$ for all $1 \leq i \leq k$, and
\begin{equation}
	\label{eq:yi-epsilon-comparison}
	0 \leq y_i^{\epsilon} \leq \tilde{y}_i^{\epsilon} \quad \text{a.e. in } \Omega.
\end{equation} 
Setting now ${y}^{k, \epsilon} := \sum_{i=1}^k \omega_i y_i^{\epsilon}$ and  $\tilde{y}^{k, \epsilon} := \sum_{i=1}^k \omega_i \tilde y_i^{\epsilon}$ 
and employing the fact that 
$\phi_\epsilon(\cdot - x_i)$ converges weakly-star to $\delta_{x_i}$ in $\mathcal{M}(\Omega)$ as $\epsilon \to 0$, we conclude that
\begin{equation}
	\label{eq:yi-epsilon-limit}
	{y}^{k, \epsilon} \to y_k \quad \text{in } W^{1,q}_0(\Omega) \quad \text{and a.e. in } \Omega;
\end{equation}
see, e.g. \cite[Thm.~2.1]{CasasKunisch2014}.
Besides, \eqref{eq:yi-epsilon-comparison} gives
\begin{equation}
	\label{eq:yi-epsilon-order}
	\tilde{y}^{k, \epsilon} \geq {y}^{k, \epsilon} \geq 0 \quad \text{a.e. in } \Omega, \quad \text{for all } \epsilon \in (0, \epsilon_0).
\end{equation}

Fixing $\alpha \in (0, 4 \pi)$ and $i \in \{1,2, \ldots, k\}$, we now exploit \cref{lem:Poisson-regularized-point-source} for the case where $m := \frac{(4 \pi - \alpha)\omega_i}{\bm\omega_{\max}}$, $\rho_0 := \rho_i$, $x_0 := x_i$,  and $y_0 := \tilde{y}_i^{ \epsilon}$ to derive 
\begin{equation}
	\label{eq:yi-epsilon-esti-pointwise-outside}
	\exp\left[\frac{(4 \pi - \alpha)\omega_i}{\bm\omega_{\max}}|\tilde{y}_i^{\epsilon}(x)|\right] \leq  \left( \frac{2R}{\rho_i -\epsilon}\right )^{(2 - \frac{\alpha}{2 \pi}) \frac{\omega_i}{\bm\omega_{\max}} }
\end{equation}
for a.e. $x \in B_R \backslash \overline{  B_{\R^2}(x_i, \rho_i)}$ and 
\begin{equation}
	\label{eq:yi-epsilon-esti-inball}
	\int_{ B_{\R^2}(x_i, \rho_i)} \exp\left[\frac{(4 \pi - \alpha)\omega_i}{\bm\omega_{\max}}|\tilde{y}_i^{ \epsilon}(x)|\right]  dx 
	\leq \frac{2 \pi (\epsilon + \rho_i)^2}{2 - {(2 - \frac{\alpha}{2 \pi}) \frac{\omega_i}{\bm\omega_{\max}} }}  \left( \frac{2R}{\epsilon + \rho_i} \right)^{(2 - \frac{\alpha}{2 \pi}) \frac{\omega_i}{\bm\omega_{\max}} }.
\end{equation}
By definition of $y^{k, \epsilon}$, we conclude that
\begin{multline}
	\label{eq:yk-esti1}
	\int_{\Omega} \exp\left[\frac{(4 \pi - \alpha)}{\bm\omega_{\max}}{y}^{k, \epsilon}(x)\right]  dx  = \int_{\Omega} \prod_{i=1}^{k} \exp\left[\frac{(4 \pi - \alpha)\omega_i}{\bm\omega_{\max}} {y}_i^{\epsilon}(x)\right]  dx \\
	\leq \int_{B_R} \prod_{i=1}^{k} \exp\left[\frac{(4 \pi - \alpha)\omega_i}{\bm\omega_{\max}} \tilde {y}_i^{ \epsilon}(x)\right]  dx,
\end{multline}
where the last inequality is derived by combining  \eqref{eq:yi-epsilon-comparison} with the fact that $\Omega \subset B_R$.
Setting $B_0 := \cup_{i=1}^k B_{\R^2}(x_i, \rho_i)$, the right-hand side of the above estimate can be split into two terms as follows
\begin{multline}
	\label{eq:yk-esti2}
	\int_{B_R\backslash B_0} \prod_{i=1}^{k} \exp\left[\frac{(4 \pi - \alpha)\omega_i}{\bm\omega_{\max}} \tilde {y}_i^{ \epsilon}(x)\right]  dx \\
	+ \sum_{l=1}^{k} \int_{B_{\R^2}(x_l, \rho_l)} \prod_{i=1}^{k} \exp\left[\frac{(4 \pi - \alpha)\omega_i}{\bm\omega_{\max}} \tilde {y}_i^{ \epsilon}(x)\right]  dx =: A_1^{k, \epsilon} + A_2^{k, \epsilon}.
\end{multline}
For $A_1^{k, \epsilon}$, by using \eqref{eq:yi-epsilon-esti-pointwise-outside}, we have
\begin{align*}
	A_1^{k, \epsilon} & \leq \int_{B_R\backslash B_0} dx \prod_{i=1}^{k}   \left( \frac{2R}{\rho_i -\epsilon} \right)^{(2 - \frac{\alpha}{2 \pi}) \frac{\omega_i}{\bm\omega_{\max}} } 
	= \pi \prod_{i=1}^{k}  \left( \frac{2R}{\rho_i -\epsilon} \right)^{(2 - \frac{\alpha}{2 \pi})  \frac{\omega_i}{\bm\omega_{\max}} } [R^2 - \sum_{i=1}^{k} \rho_i^2]\\
	& = \pi [R^2 - \sum_{i=1}^{k} \rho_i^2]\left( 2R \right)^{(2 - \frac{\alpha}{2\pi}) \frac{\sum_{i=1}^k \omega_i}{\bm{\omega}_{\max}} }\prod_{i=1}^{k}  \left( \frac{1}{\rho_i -\epsilon} \right)^{(2 - \frac{\alpha}{2 \pi})  \frac{\omega_i}{\bm\omega_{\max}} }.
\end{align*}
Consequently, one has
\begin{equation}
	\label{eq:A1k-lim}
	\limsup\limits_{\epsilon \to 0^+} A_1^{k, \epsilon} \leq
	\pi [R^2 - \sum_{i=1}^{k} \rho_i^2](2R)^{(2 - \frac{\alpha}{2\pi}) \frac{\sum_{i=1}^k \omega_i}{\bm{\omega}_{\max}} }\prod_{i=1}^{k}  \left( \frac{1}{\rho_i} \right)^{(2 - \frac{\alpha}{2 \pi})  \frac{\omega_i}{\bm\omega_{\max}} }.
\end{equation}
For $A_2^{k, \epsilon}$, in light of \eqref{eq:yi-epsilon-esti-pointwise-outside}  and \eqref{eq:yi-epsilon-esti-inball}, we obtain
\begin{align*}
	A_2^{k, \epsilon} & = \sum_{l=1}^{k} \int_{B_{\R^2}(x_l, \rho_l)} \exp\left[\frac{(4 \pi - \alpha)\omega_l}{\bm\omega_{\max}} \tilde {y}_l^{ \epsilon}(x)\right]  \times   \prod_{i \neq l}  \exp\left[\frac{(4 \pi - \alpha)\omega_i}{\bm\omega_{\max}} \tilde {y}_i^{ \epsilon}(x)\right]  dx \\
	& \leq \sum_{l=1}^{k} \prod_{i \neq l}   \left( \frac{2R}{\rho_i -\epsilon} \right)^{(2 - \frac{\alpha}{2 \pi})  \frac{\omega_i}{\bm\omega_{\max}} } \int_{B_{\R^2}(x_l, \rho_l)} \exp\left[\frac{(4 \pi - \alpha)\omega_l}{\bm\omega_{\max}} \tilde {y}_l^{\epsilon}(x)\right]dx \\
	& \leq 
	\sum_{l=1}^{k} \prod_{i \neq l}   \left( \frac{2R}{\rho_i -\epsilon} \right)^{(2 - \frac{\alpha}{2 \pi})  \frac{\omega_i}{\bm\omega_{\max}} }
	\times \frac{2 \pi (\epsilon + \rho_l)^2}{2 - {(2 - \frac{\alpha}{2 \pi}) \frac{\omega_l}{\bm\omega_{\max}} }} \times \left( \frac{2R}{\epsilon + \rho_l} \right)^{(2 - \frac{\alpha}{2 \pi}) \frac{\omega_l}{\bm\omega_{\max}} } \\
	& = 	
	2 \pi (2R)^{(2 - \frac{\alpha}{2 \pi})   \frac{\sum_{i=1}^k \omega_i}{\bm\omega_{\max}} }
	\sum_{l=1}^k (\epsilon + \rho_l)^2\left(\frac{\rho_l - \epsilon}{\rho_l + \epsilon}\right)^{(2 - \frac{\alpha}{2 \pi}) \frac{\omega_l}{\bm\omega_{\max}} } \left[{2 - {\left(2 - \frac{\alpha}{2 \pi}\right) \frac{\omega_l}{\bm\omega_{\max}} }}\right]^{-1} \\
	& \qquad \times \prod_{i =1}^k   \left( \frac{1}{\rho_i -\epsilon}\right)^{(2 - \frac{\alpha}{2 \pi})  \frac{\omega_i}{\bm\omega_{\max}} }.
\end{align*}  
Then, there holds
\begin{multline}
	\label{eq:A2k-lim}
	\limsup\limits_{\epsilon \to 0^+} A_2^{k, \epsilon}  \\
	\begin{aligned}[b]
		&\leq 2 \pi (2R)^{(2 - \frac{\alpha}{2 \pi})   \frac{\sum_{i=1}^k \omega_i}{\bm\omega_{\max}} }
		\sum_{l=1}^k \rho_l^2 \left[{2 - {\left(2 - \frac{\alpha}{2 \pi}\right) \frac{\omega_l}{\bm\omega_{\max}} }}\right]^{-1} \times \prod_{i =1}^k   \left( \frac{1}{\rho_i} \right)^{(2 - \frac{\alpha}{2 \pi})  \frac{\omega_i}{\bm\omega_{\max}} }  \\
		& \leq   \frac{4\pi^2}{\alpha} (2R)^{(2 - \frac{\alpha}{2 \pi})   \frac{\sum_{i=1}^k \omega_i}{\bm\omega_{\max}} }
		\sum_{l=1}^k \rho_l^2 \prod_{i =1}^k   \left( \frac{1}{\rho_i} \right)^{(2 - \frac{\alpha}{2 \pi})  \frac{\omega_i}{\bm\omega_{\max}} },
	\end{aligned}
\end{multline}
where we have just employed the estimates $[{2 - {(2 - \frac{\alpha}{2 \pi}) \frac{\omega_l}{\bm\omega_{\max}} }}]^{-1} \leq \frac{2\pi}{\alpha}$ for all $1 \leq l \leq k$ to get the last inequality.
From \eqref{eq:A1k-lim} and \eqref{eq:A2k-lim}, it follows that
\begin{multline*}
	\limsup\limits_{\epsilon \to 0^+} [A_1^{k, \epsilon}+A_2^{k, \epsilon}] \\
	\leq \pi  (2R)^{(2 - \frac{\alpha}{2\pi}) \frac{\sum_{i=1}^k \omega_i }{\bm{\omega}_{\max}} }\prod_{i=1}^{k}  \left( \frac{1}{\rho_i} \right)^{(2 - \frac{\alpha}{2 \pi})  \frac{\omega_i}{\bm\omega_{\max}} }\left[R^2 +\left(\frac{4\pi}{\alpha} -1\right) \sum_{i=1}^{k} \rho_i^2\right].
\end{multline*}
Combining this with \eqref{eq:yk-esti1} and \eqref{eq:yk-esti2} yields
\begin{multline*}
	\limsup\limits_{\epsilon \to 0^+} \int_{\Omega} \exp\left[\frac{(4 \pi - \alpha)}{\bm\omega_{\max}}{y}^{k, \epsilon}(x)\right]  dx \\
	\begin{aligned}
		& \leq \pi  (2R)^{(2 - \frac{\alpha}{2\pi}) \frac{\sum_{i=1}^k \omega_i}{\bm{\omega}_{\max}} }\prod_{i=1}^{k}  \left( \frac{1}{\rho_i} \right)^{(2 - \frac{\alpha}{2 \pi})  \frac{\omega_i}{\bm\omega_{\max}} }\left[R^2 +\left(\frac{4\pi}{\alpha} -1\right) \sum_{i=1}^{k} \rho_i^2\right].
	\end{aligned}
\end{multline*}
This, along with \eqref{eq:yi-epsilon-limit}, the nonnegativity of $y^{k, \epsilon}$, and Fatou's lemma, yields \eqref{eq:exponential-esti-Dirac-k}.
\end{proof}

\begin{proposition}
	\label{prop:exponential-esti-extended}
	Let 
	$\alpha \in (0, 4\pi)$,
	$\psi \in L^p(\Omega)$ with $p>1$, and let $\bm{\omega} \in \ell^1\backslash \{\bm{0}\}$ satisfy  $\omega_i \geq 0$ for all $i \geq 1$ and \eqref{eq:omega-rho-ass}.
	Assume that $g: \Omega \times \R \to \R$ is a Carathéodory function such that, for a.e. $x \in \Omega$, the mapping $ t \mapsto g(x,t)$ is nondecreasing and continuous,  and satisfies $g(x,0) =0$.
	Assume further that $y \in W^{1,q}_0(\Omega)$, $1 \leq q <2$, is a weak solution to
	\begin{equation}
		\label{eq:semilinear-Dirac}
		-\Delta y + g(x,y)  = \psi(x) + \sum_{i=1}^\infty \omega_i \delta_{x_i} \, \text{in } \Omega, \quad
		y = 0 \, \text{on } \Gamma.
	\end{equation}
		Then
		a constant $C>0$, independent of $\alpha$, $g, \psi$, and $\bm{\omega}$, exists and satisfies 
		\begin{multline}
			\label{eq:exponential-esti-Dirac-semilinear}
			\int_\Omega \exp \left[\frac{(4\pi - \alpha)y(x)}{\bm\omega_{\max}} \right] dx \leq \frac{4\pi^2 R^2}{\alpha} (2R)^{(2 - \frac{\alpha}{2\pi}) \frac{\norm{\bm{\omega}}_{\ell^1}}{\bm{\omega}_{\max}} } \\
			\times
			\exp\left\{\frac{1}{\bm\omega_{\max}}\left(2 - \frac{\alpha}{2 \pi}\right)\left[C\norm{\psi}_{L^p(\Omega)} +   L(\bm{\omega}, \bm{\rho})\right]  \right\}
		\end{multline}
		with $R := \frac{1}{2} \diam(\Omega)$.
\end{proposition}
\begin{proof}
Let $y_0$ and $y_1$ be the unique solutions to 
\begin{align*}
	& 	\left\{
	\begin{aligned}
		-\Delta y_0 + g(x,y_0) & = \psi  && \text{in } \Omega, \\
		y_0 &= 0 && \text{on } \Gamma
	\end{aligned}
	\right. 
	\quad
	\text{and} \quad
	\left\{
	\begin{aligned}
		-\Delta y_1 & =  \sum_{i=1}^\infty \omega_i \delta_{x_i} && \text{in } \Omega, \\
		y_1 &= 0 && \text{on } \Gamma,
	\end{aligned}
	\right.
\end{align*}	
respectively. 	
Because of the fact that $\psi \in L^p(\Omega)$ with $p> 1 = \frac{N}{2}$, we obtain $y_0 \in C(\overline{\Omega})$ and
\begin{equation}
	\label{eq:y0-Linfty}
	\norm{y_0}_{L^\infty(\Omega)} \leq C \norm{\psi}_{L^p(\Omega)}
\end{equation}
for some constant $C> 0$; see, e.g. \cite[Thm.~4.7]{Troltzsch2010}. 
Moreover, by \Cref{prop:Poisson-exponential-esti-inf-many-Dirac}, one has
\begin{equation}
	\label{eq:exponential-y1}
	\int_\Omega \exp \left[\frac{(4\pi - \alpha)y_1(x)}{\bm\omega_{\max}} \right] dx \leq \frac{4\pi^2 R^2}{\alpha} (2R)^{(2 - \frac{\alpha}{2\pi}) \frac{\norm{\bm{\omega}}_{\ell^1}}{\bm{\omega}_{\max}} } 
	\exp\left[\frac{1}{\bm\omega_{\max}}\left(2 - \frac{\alpha}{2 \pi}\right)   L(\bm{\omega}, \bm{\rho})\right].
\end{equation}
On the other hand, from the comparison principle; see, e.g. \cite[Cor.~4.B.2]{BresisMarcusPonce2007}\footnote{Although \cite[Cor.~4.B.2]{BresisMarcusPonce2007} considers the case where $g$ does not depend on $x$, it remains valid for the more general case $g= g(x,t)$}, and the nonnegativity of all $\omega_i$, $i \geq 1$, there holds $y \geq y_0$ a.e. in $\Omega$. From this and the nondecreasing property of the mapping $t \mapsto g(x,t)$ for a.e. $x \in \Omega$, we have $g(x,y(x)) \geq g(x,y_0(x))$ a.e. $x \in \Omega$. Consequently, one has
\[
	- g(x,y) + \psi + \sum_{i=1}^\infty \omega_i \delta_{x_i} \leq -g(x,y_0) + \psi  + \sum_{i=1}^\infty \omega_i \delta_{x_i}.
\]
Combining this with the comparison principle, it follows that
\[
	y \leq y_0 + y_1 \quad \text{a.e. in } \Omega.
\]
This, together with \eqref{eq:y0-Linfty} and \eqref{eq:exponential-y1}, gives \eqref{eq:exponential-esti-Dirac-semilinear}.
\end{proof}

\medskip
In the remainder of this section, we establish the existence and uniqueness of solutions to the linearized state equation. This result will be used in the next section to define the generalized differentiability of the control-to-state mapping.
\begin{lemma}[{cf. \cite[Lem.~3.1]{Nhu2025}}]
	\label{lem:linearization-state-equation}
	Assume that $y \in L^1(\Omega)$ satisfying $e^y \in L^r(\Omega)$ with $r >1$. Let $\bm{u} = (u_i) \in \ell^1$ and $\psi \in L^1(\Omega)$ be arbitrary but fixed. Then, there exists a unique $z \in W^{1,q}_0(\Omega)$ for all $1 \leq q < 2$ satisfying
	\begin{equation}
		\label{eq:linearization-state-equation}
		-\Delta z + e^y z  = \psi + \sum_{i=1}^\infty u_i \delta_{x_i} \, \text{in } \Omega \quad
		z = 0 \, \text{on } \Gamma.
		%
	\end{equation}
	Furthermore, a constant $C>0$ independent of $y$, $\psi$, and $\bm{u}$ exists and satisfies
	\begin{equation}
		\label{eq:linearization-state-equation-a-priori-esti}
		\norm{z}_{W^{1,q}_0(\Omega)} \leq C [\norm{\psi}_{L^1(\Omega)} + \norm{\bm{u}}_{\ell^1} ].
	\end{equation}
\end{lemma}
\begin{proof}
	The proof follows the same argument as that of \cite[Lem.~3.1]{Nhu2025}, with $\bm{u} \in \ell^1$ replacing $\bm{\eta} \in \mathbb{R}^k$ throughout.
\end{proof}

\section{Finite-dimensional directional differentiability and generalized differentiability}
	\label{sec:differentiability}
	In this section, we first show that the control-to-state operator $S$, which maps each $\bm{u}$ in a certain subset of $\ell^1$ to the associated solution of \eqref{eq:P-state}, is Lipschitz continuous.	
	We then prove that $S$ is directionally differentiable in  directions belonging to finite-dimensional subspaces $\ell^1_k$, $k \geq 1$, at every element $\bm{u}$ for which the exponential Nemytskii superposition operator of $S(\bm{u})$ takes values in $L^r(\Omega)$ for some $r>1$.
	Based on this, we define a generalized directional derivative of $S$ in any direction $\bm{h} \in \ell^1$ as the limit of directional derivatives along the $k$-th truncated directions $\bm{h}_k$, $k \geq 1$. 
	Finally, we investigate the first- and second-order generalized differentiability of the reduced cost functional.
	These notions will be used to derive the optimality conditions for \eqref{eq:P} in the next section.
	
	\subsection{Control-to-state mapping and its properties}
	\label{sec:control2state-mapping}
	
	We now recall the state equation \eqref{eq:P-state}:
	\begin{equation}
		\tag{\ref{eq:P-state}}
		\left\{
		\begin{aligned}
			-\Delta y + \left( e^y -1  \right)  & = f_0 + \sum_{i=1}^\infty u_i \delta_{x_i} && \text{in } \Omega \\
			y &= 0 && \text{on } \Gamma
		\end{aligned}
		\right.
	\end{equation}
	with $\bm{u} = (u_i) \in \ell^1$, $f_0 \in L^p(\Omega)$, $p >1$.  
	
	A function $y \in L^1(\Omega)$ for all $1 \le q < 2$ is called a \emph{(very) weak solution} of \eqref{eq:P-state} if $e^y \in L^1(\Omega)$ and
	\begin{equation}
		\int_\Omega  - y \Delta \phi + (e^y-1)\phi dx = \int_\Omega f_0 \phi dx + \sum_{i =1}^\infty u_i \phi(x_i) \quad \text{for all } \phi \in C^\infty_0(\Omega). 
	\end{equation}
	
	By \cite[Thm.~2]{Vazquez1983} (see also \cite[Thm.~4.7]{BresisMarcusPonce2007} and \cite{Bartolucci2005}), a necessary and sufficient condition for the existence of (very) weak solutions to \eqref{eq:P-state} is 
	\begin{equation}
		\label{eq:SE-existence-solution-assumption}
		\bm{u}_{\max} \leq 4\pi.
	\end{equation} 
	Furthermore, under \eqref{eq:SE-existence-solution-assumption}, the solution to \eqref{eq:P-state} is unique and it belongs to $W^{1,q}_0(\Omega)$ for any $1 \leq q < \frac{N}{N-1} =2$. This leads to the definition of the control-to-state mapping $S$ as follows
	\begin{equation}
		\label{eq:control-2-state-oper}
		\begin{aligned}
			S: \mathcal{D} &\to   W^{1,q}_0(\Omega)\\
			\bm{u} &\mapsto y_{\bm{u}},
		\end{aligned}
	\end{equation}
	where 
	\[
		\mathcal{D} := \left\{ \bm{u} = (u_i) \in \ell^1 \middle| u_i \leq 4 \pi \, \text{for all } i \geq 1 \right\}
	\]
	and $y_{\bm{u}}$ denotes the unique weak solution to \eqref{eq:P-state} associated with $\bm{u}$.

	The Lipschitz continuity of $S$ and of $\exp(S)$ is shown in the next lemma.
\begin{lemma}[{cf. \cite[Lem.~4.1]{Nhu2025} }]
	\label{lem:control2state-oper-continuity}
	For any $q \in [1,2)$, a constant $C = C(q)>0$ exists and satisfies
	\begin{equation}
		\label{eq:SE-apriori-esti}
		\norm{S(\bm{u})}_{W^{1,q}_0(\Omega)} \leq C \left( \norm{f_0}_{L^p(\Omega)} + \norm{\bm{u}}_{\ell^1} \right)
	\end{equation}
	for all $\bm{u} \in \mathcal{D}$. Moreover, there holds
	\begin{equation}
		\label{eq:exp-S-esti-L1}
		\int_{\Omega} \left|e^{S(\bm{u})} -1 \right| dx \leq |\Omega|^{\frac{p-1}{p}} \norm{f_0}_{L^p(\Omega)} + \norm{\bm{u}}_{\ell^1}.
	\end{equation}
	Furthermore, there hold
	\begin{align}
		& \int_{\Omega} \left[ e^{S(\bm{u})} - e^{S(\bm{v})} \right]^{+} dx \leq \sum_{i: u_i \geq v_i} (u_i - v_i), \label{eq:exp-S-pos-continuity-L1} \\
		& \int_{\Omega} \left|e^{S(\bm{u})} - e^{S(\bm{v})} \right| dx \leq \norm{\bm{u} - \bm{v}}_{\ell^1} \label{eq:exp-S-continuity-L1}
		\intertext{and}
		& \norm{S(\bm{u}) - S(\bm{v})}_{W^{1,q}_0(\Omega)} \leq C \norm{\bm{u} - \bm{v}}_{\ell^1} \label{eq:S-continuity-W1q}
	\end{align}
	for all $\bm{u}, \bm{v} \in \mathcal{D}$. 
\end{lemma}	
\begin{proof}
	The proof of the lemma is similar to that of \cite[Lem.~4.1]{Nhu2025} with minor modifications.
\end{proof}
	
	The following lemma shows that the composition $\exp(S)$ of the control-to-state operator with the exponential Nemytskii superposition operator takes values in $L^r(\Omega)$ for some $r>1$ at points satisfying \eqref{eq:omega-rho-ass}.
\begin{lemma} 
	\label{lem:expS-esti}
	Assume that $\bm{u} = (u_i) \in \mathcal{D}$ satisfies \eqref{eq:omega-rho-ass} with $\bm{\omega} := \bm{u}$ and
	\begin{equation}
		\label{eq:ui-less-4pi}
		u_i < 4 \pi \quad \text{for all } i \geq 1.
	\end{equation}
	There, there exists a constant $r := r(\bm{u}) >1$ such that
	\begin{equation}
		\label{eq:exp-S-belong-Lr}
		e^{S(\bm{u})} \in L^r(\Omega).
	\end{equation}
	Moreover, there holds
	\begin{equation}
		\label{eq:exp-S-belong-Lr-order}
		e^{S(\bm{v})} \in L^r(\Omega) \quad \text{for all } \bm{v} = (v_i) \in \mathcal{D}, v_i \leq u_i, i \geq 1.
	\end{equation}
\end{lemma}
\begin{proof}
	It suffices to prove \eqref{eq:exp-S-belong-Lr}, since \eqref{eq:exp-S-belong-Lr-order} follows from \eqref{eq:exp-S-belong-Lr} and the comparison principle; see, e.g. \cite[Cor.~4.B.2]{BresisMarcusPonce2007}.	
	Since $\bm{u} \in \ell^1$, $u_i \to 0$ as $i \to \infty$. From this  and \eqref{eq:ui-less-4pi}, there exists an integer $i_0 \geq 1$  such that
	\[
		\max \left\{ u_i \middle | i \geq 1 \right\} = u_{i_0} < 4 \pi.
	\]
	Let us put $y:= S(\bm{u})$ and let $y_0 := S(\bm{0})$ be the unique solution to
	\begin{equation*}
		\left\{
		\begin{aligned}
			-\Delta y_0 + \left( e^{y_0} -1  \right)  & = f_0  && \text{in } \Omega \\
			y_0 &= 0 && \text{on } \Gamma.
		\end{aligned}
		\right.
	\end{equation*}
	By Theorem 4.7 in \cite{Troltzsch2010} and the fact that $f_0 \in L^p(\Omega)$ with $p>1$, one has $y_0 \in C(\overline{\Omega})$ and
	\begin{equation}
		\label{eq:y0-bounded}
		\norm{y_0}_{C(\overline{\Omega})} \leq C
	\end{equation}
	for some constant $C>0$.

	We now consider two cases:
	
	\noindent\emph{Case I: $u_{i_0} \leq 0$.} In this case, there holds $u_i \leq 0$ for all $i \geq 1$. The comparison; see, e.g. \cite[Cor.~4.B.2]{BresisMarcusPonce2007}, implies that $y \leq y_0$ a.e.
	in $\Omega$. Thanks to \eqref{eq:y0-bounded}, we have  $e^y \in L^\infty(\Omega) \hookrightarrow L^r(\Omega)$ for any $r >1$.

	\noindent\emph{Case II: $u_{i_0} \in (0, 4 \pi)$.} Setting now $y_{+} := S(\bm{u}^+)$ with $\bm{u}^{+} := (u_i^{+})$ and using the comparison principle; see, e.g. \cite[Cor.~4.B.2]{BresisMarcusPonce2007} yield
	\begin{equation}
		\label{eq:y-y0-yplus-comparison}
		y, y_0 \leq y_{+} \quad \text{a.e. in } \Omega.
	\end{equation}
	This implies that
	\begin{equation}
		\label{eq:exp-y0-yplus-comparison}
		e^{y_{+}} - e^{y_0} \geq 0 \quad \text{a.e. in } \Omega.
	\end{equation}
	By subtracting the equations for $y_{+}$ and $y_0$, we obtain
	\[
		\left\{
		\begin{aligned}
			-\Delta (y_{+} - y_0) + \left( e^{y_{+}} -e^{y_0}  \right)  & = \sum_{i=1}^\infty u_i^+ \delta_{x_i}  && \text{in } \Omega \\
			y_{+} - y_0 &= 0 && \text{on } \Gamma.
		\end{aligned}
		\right.
	\]
	The comparison principle, together with \eqref{eq:exp-y0-yplus-comparison},  gives 
	$y_{+} - y_0 \leq z$ for a.e.  $\Omega$. Here  $z$ denotes the unique solution to 
	\[
		\left\{
		\begin{aligned}
			-\Delta z  & = \sum_{i=1}^\infty u_i^+ \delta_{x_i}  && \text{in } \Omega \\
			z &= 0 && \text{on } \Gamma.
		\end{aligned}
		\right.
	\]
	This, \eqref{eq:y-y0-yplus-comparison} and \eqref{eq:y0-bounded} lead to
	\begin{equation}
		\label{eq:y-y0-z-esti}
		y \leq y_{+} \leq y_0 + z \leq C + z  \quad \text{a.e. in } \Omega.
	\end{equation}
	Moreover, by fixing $r \in (1, \frac{4\pi}{u_{i_0}})$ and applying \cref{prop:Poisson-exponential-esti-inf-many-Dirac} to the equation for $z$ with $\alpha := 4\pi - r u_{i_0} \in (0,4\pi)$, one has  
	\begin{equation}
		\label{eq:ey-Ltau}
		\norm{e^z}_{L^r(\Omega)}^r < \infty,
	\end{equation}
	where we have just used the assumption \eqref{eq:omega-rho-ass} for  $\bm{\omega} := \bm{u}$  and the estimate \eqref{eq:exponential-esti-Dirac-inf}.
\end{proof}

\subsection{Directional differentiability of the control-to-state mapping}
\label{sec:control2state-generalized-diff}

At the beginning of this subsection, we establish the directional differentiability of the control-to-state operator with respect to directions in finite-dimensional subspaces $\ell^1_k$, $k \geq 1$, of $\ell^1$. This result, together with the chain rule, will be used in the next subsection to prove the same property for the reduced objective functional.

\begin{theorem}
	\label{thm:directional-diff-control2state}
	Let $\bm{u} \in \mathcal{D}$ be such that $e^{S(\bm{u})} \in L^r(\Omega)$ for some constant $r >1$ and let $k \geq 1$ be an arbitrary, but fixed integer. 
	Then, $S$ is directional differentiable at $\bm{u}$ in any direction $\bm{h} \in \ell^1_k$ and $z:= S'(\bm{u})\bm{h}$ uniquely solves
	\begin{equation}
		\label{eq:1st-directional-deri-S}
		-\Delta z + e^{S(\bm{u})} z  = \sum_{i=1}^{k} h_i \delta_{x_i} \, \text{in } \Omega \quad
		z = 0 \, \text{on } \Gamma.
		%
	\end{equation}
	Moreover, there hold for any $s \geq 1$ that
	\begin{align}
		& z_\rho \to z \quad \text{as } \rho \to 0^+ \quad \text{in } L^s(\Omega), \label{eq:zrho-z-limit} \\
		& \frac{e^{\rho z_\rho} - 1 - \rho z_\rho }{\rho} \to 0 \quad \text{as } \rho \to 0^+ \quad \text{in } L^{s}(\Omega) \label{eq:expS-1st-Taylor}
		\intertext{and}
		& \frac{e^{\rho z_\rho} - 1 - \rho z_\rho - \frac{1}{2}\rho^2 z_\rho^2 }{\rho^2} \to 0 \quad \text{as } \rho \to 0^+ \quad \text{in } L^{s}(\Omega) \label{eq:expS-2st-Taylor}	
	\end{align}
	with $z_\rho := \frac{S(\bm{u} + \rho \bm{h}) - S(\bm{u})}{\rho}$. 
\end{theorem}
\begin{proof}
	By setting $y_\rho := S(\bm{u} + \rho \bm{h})$ and $y := S(\bm{u})$ and subtracting the equations for $y_\rho$ and $y$, we have	
	\[
		\left\{
			\begin{aligned}
				- \Delta (y_\rho - y) + e^{y_\rho} - e^y & = \rho \sum_{i=1}^k h_i \delta_{x_i} && \text{in } \Omega,\\
				y_\rho - y & = 0 && \text{on } \Gamma,
			\end{aligned}
		\right.
	\]
	or, equivalently,
	\begin{equation}
		\label{eq:zrho-def}
		\left\{
		\begin{aligned}
			- \Delta z_\rho + e^y\frac{e^{\rho z_\rho} - 1}{\rho} & =  \sum_{i=1}^k h_i \delta_{x_i} && \text{in } \Omega,\\
			z_\rho & = 0 && \text{on } \Gamma.
		\end{aligned}
		\right.
	\end{equation}	
	The comparison principle thus implies that
	\begin{equation}
		\label{eq:zrho-zrho-hat}
		z_\rho \leq \hat z_\rho \quad \text{a.e. in } \Omega,
	\end{equation}
	where $\hat z_\rho$ stands for the unique solution to
	\begin{equation}
		\label{eq:hat-zrho}
		\left\{
		\begin{aligned}
			- \Delta \hat z_\rho + e^y\frac{e^{\rho \hat z_\rho} - 1}{\rho} & =  \sum_{i=1}^k h_i^{+} \delta_{x_i} && \text{in } \Omega,\\
			\hat z_\rho & = 0 && \text{on } \Gamma.
		\end{aligned}
		\right.
	\end{equation}
	Clearly, we have $\hat{z}_\rho = \frac{1}{\rho}\left(S(\bm{u} + \rho \bm{h}^{+}) - S(\bm{u}) \right)$ with
	\[
		\bm{h}^{+} := (h_1^+, h_2^+, \ldots, h_k^+,0,0,\ldots) \in \ell^1_k.
	\]
	If $h_i^{+} = 0$ for all $1 \leq i \leq k$, then $\hat z_\rho = 0$. Consequently, \eqref{eq:zrho-zrho-hat} leads to
	\[
		e^{ z_\rho } \in L^\infty(\Omega).
	\]
	Otherwise, we have $\max\{ h_i^+ | 1 \leq i \leq k \} >0$.  
	Applying now \cref{prop:exponential-esti-extended} to \eqref{eq:hat-zrho} yields
	\[
		e^{\rho_0 \hat z_\rho } \in L^1(\Omega) \quad \text{for some } \rho_0 >0.
	\]
	In all cases, we therefore have
	\begin{equation}
		\label{eq:exp-zrho-in-L1}
		e^{\rho_0 z_\rho } \in L^1(\Omega) \quad \text{for some } \rho_0 >0.
	\end{equation}
	On the other hand, by exploiting \eqref{eq:S-continuity-W1q} in \cref{lem:control2state-oper-continuity}, we deduce that
	\[
		\norm{z_{\rho}}_{W^{1,q}_0(\Omega)} \leq C \norm{\bm{h}}_{\ell^1}
	\]
	for some constant $C>0$ and for all $1 \leq q < 2$. The compact embedding $W^{1,q}_0(\Omega) \hookrightarrow L^{s + 1}(\Omega)$ for some $q$ sufficiently close $2$ then indicates that
	\begin{equation}
		\label{eq:zrho-Lr-bounded}
		\norm{z_{\rho}}_{L^{s+1}(\Omega)} \leq C \norm{\bm{h}}_{\ell^1}
	\end{equation}  
	and that
	\begin{equation}
		\label{eq:zrho-limit}
		z_{\rho'} \rightharpoonup z \quad \text{in } W^{1,q}_0(\Omega), \quad  z_{\rho'} \to z \quad \text{in } L^{s+1}(\Omega), \quad \text{and} \quad z_{\rho'} \to z \quad \text{a.e. in } \Omega
	\end{equation}
	for some subsequence $\{\rho'\}$ of $\{\rho \}$ and $z \in W^{1,q}_0(\Omega)$.
	We now show that
	\begin{equation}
		\label{eq:exp-zrho-limit}
		\frac{e^{\rho'  z_{\rho'}} - 1 - \rho' z_{\rho'} }{\rho'} \to 0 \quad \text{in } L^{s}(\Omega),
	\end{equation}
	by applying \eqref{eq:expo-est-1st} and \cite[Chap.~XI, Prop.~3.10]{Visintin1996}. To this end, we conclude from \eqref{eq:expo-est-1st} that
	\[
		 0 \leq A_{\rho'} := \frac{e^{\rho'  z_{\rho'}} - 1 - \rho' z_{\rho'} }{\rho'} \leq \frac{e^{\rho_1  z_{\rho'}} - 1 - \rho_1 z_{\rho'} }{\rho_1}  \quad \text{for all } 0 < \rho' < \rho_1 
	\]
	and for a.e. in $\Omega$ with $\rho_1 := \frac{\rho_0}{s+1}$. Therefore, $\{A_{\rho'}\}$ is bounded in $L^{s+1}(\Omega)$ due to \eqref{eq:exp-zrho-in-L1} and \eqref{eq:zrho-Lr-bounded}.
	Moreover, according to \eqref{eq:zrho-limit}, $A_{\rho'}$ converges in measure to $0$ as $\rho' \to 0$. Thanks to \cite[Chap.~XI, Prop.~3.10]{Visintin1996}, there holds  $A_{\rho'} \to 0$ in $L^{s}(\Omega)$. We then have \eqref{eq:exp-zrho-limit}.
	
	In \eqref{eq:zrho-def}, we choose $\rho := \rho'$. By taking $\rho' \to 0^+$ and using limits \eqref{eq:zrho-limit} as well as \eqref{eq:exp-zrho-limit}, we deduce that $z$ satisfies \eqref{eq:1st-directional-deri-S}. The uniqueness of solutions to \eqref{eq:1st-directional-deri-S} (see \cref{lem:linearization-state-equation}) implies that the limits \eqref{eq:zrho-limit} and \eqref{eq:exp-zrho-limit} hold for the entire sequence $\{\rho\}$, rather than only for the subsequence $\{\rho' \}$. 
	This proves \eqref{eq:zrho-z-limit} and \eqref{eq:expS-1st-Taylor}.
	Furthermore, subtracting \eqref{eq:1st-directional-deri-S} from \eqref{eq:zrho-def} yields
	\begin{equation*}
		\left\{
		\begin{aligned}
			- \Delta (z_\rho - z)  & =  - e^y \left(\frac{e^{\rho z_\rho} - 1}{\rho} - z\right) && \text{in } \Omega,\\
			z_\rho - z & = 0 && \text{on } \Gamma.
		\end{aligned}
		\right.
	\end{equation*}	
	From this, together with \eqref{eq:zrho-z-limit} and \eqref{eq:expS-1st-Taylor}, we have $z_\rho \to z$ in $W^{1,q}_0(\Omega)$, i.e., $S$ is directionally differentiable at $\bm{u}$ in direction $\bm(h)$.

	Finally, the proof of \eqref{eq:expS-2st-Taylor} is similar to that of \eqref{eq:expS-1st-Taylor}, except that we use \eqref{eq:expo-est-2nd} instead of  \eqref{eq:expo-est-1st}.
\end{proof}

At the end of this subsection, we introduce the notion of \emph{generalized directional differentiability} for the control-to-state operator $S$, which will be used later to define the generalized differentiability of the reduced objective functional.
\begin{definition}
	\label{def:generalized-der-S}
	For any $\bm{u} \in \mathcal{D}$ with $e^{S(\bm{u})} \in L^r(\Omega)$ for some $r > 1$, we define a \emph{generalized derivative} of $S$ at $\bm{u}$ as the mapping
	\[
		DS(\bm{u}) : \ell^1 \to W^{1,q}_0(\Omega), \quad 1 \le q < 2,
	\]
	that maps each $\bm{h} \in \ell^1$ to the unique solution $z := DS(\bm{h})$ of
	\begin{equation}
		\label{eq:generalized-deri-S}
		-\Delta z + e^{S(\bm{u})} z = \sum_{i=1}^{\infty} h_i \delta_{x_i}
		\quad \text{in } \Omega, 
		\qquad
		z = 0 \quad \text{on } \Gamma.
	\end{equation}
\end{definition}

\begin{proposition}
	\label{prop:generalized-dir-S}
	Assume that $\bm{u} \in \mathcal{D}$ such that $e^{S(\bm{u})} \in L^r(\Omega)$ for some constant $r >1$. Then, there holds
	\begin{equation}
		\label{eq:generalized-dir-S-limit}
		S'(\bm{u}) \bm{h}_k \to DS(\bm{u}) \bm{h} \quad \text{in } W^{1,q}_0(\Omega) \quad \text{as } k \to \infty.
	\end{equation}
	for all $1 \leq q < 2$ and $\bm{h} \in \ell^1$.
\end{proposition}
\begin{proof}
	The well-definedness of $DS(\bm{u})$ follows directly from \cref{lem:linearization-state-equation}. To prove the limit \eqref{eq:generalized-dir-S-limit}, we now set $z:= DS(\bm{u})\bm{h}$ and $z_k := S'(\bm{u})\bm{h}_k$, $ k\geq 1$. Subtracting the equations for $z_k$ and $z$ (see \eqref{eq:1st-directional-deri-S} and \eqref{eq:generalized-deri-S}) yields
	\[
		\left\{
		\begin{aligned}
			-\Delta (z_k - z) + e^{S(\bm{u})}(z_k - z) &= - \sum_{i=k+1}^\infty h_i \delta_{x_i} && \text{in } \Omega,\\
			z_k - z & = 0 && \text{on } \Gamma.
		\end{aligned}
		\right.
	\]
	As a consequence of \eqref{eq:linearization-state-equation-a-priori-esti}, there holds
	\[
		\norm{z_k - z}_{W^{1,q}_0(\Omega)} \leq C \sum_{i=k+1}^\infty |h_i|
	\]
	for all $1 \leq q < 2$ and $k \geq 1$. Letting $k \to \infty$ gives \eqref{eq:generalized-dir-S-limit}.
\end{proof}

\subsection{Directional differentiability of the objective functional}
\label{sec:cost-func}
In this subsection, we investigate the first- and second-order  directional differentiability of the reduced cost function $J$ of \eqref{eq:P} along directions in $\ell^1_k$, $k \geq 1$. 
This concept will be used to introduce the generalized differentiability of $J$, which will be used in the next section to derive the first- and second-order optimality conditions for the optimal control problem \eqref{eq:P}.

Owing to the finite-dimensional directional differentiability of $S$ and the chain rule, the reduced objective functional is directionally differentiable in directions belonging to $\ell^1_k$ for all $k \geq 1$, as part of the following proposition.
\begin{proposition}
	\label{prop:cost-func-1st-2nd-directional-der}
	Assume that $\bm{u} \in \mathcal{D}$ such that $e^{S(\bm{u})} \in L^r(\Omega)$ for some $r > 1$. Let $k > 1$ be an integer.
	Then the reduced cost functional $J$ is directionally differentiable at $\bm{u}$ in every direction $\bm{h} \in \ell^1_k$ and satisfies
	\begin{equation}
		\label{eq:cost-func-1st-directional-der}
		J'(\bm{u})\bm{h} = \int_\Omega (S(\bm{u}) - y_d)z_{\bm{u}; \bm{h}} dx + \nu \sum_{i=1}^k u_i h_i 
		 = \sum_{i=1}^k \varphi_{\bm{u}}(x_i) h_i + \nu \sum_{i=1}^k u_i h_i
	\end{equation}
	with $z_{\bm{u}; \bm{h}} := S'(\bm{u})(\bm{h})$ and $\varphi_{\bm{u}}$ being the unique solution  in $W^{1,s_*}_0(\Omega)$ 
	for some $s_* :=  s_*(r) >2$	to 
	\begin{equation}
		\label{eq:adjoint-state0}
		\left\{
		\begin{aligned}
			-\Delta \varphi_{\bm{u}} + e^{S(\bm{u})} \varphi_{\bm{u}}& = S(\bm{u}) - y_d && \text{in } \Omega \\
			\varphi_{\bm{u}} &= 0 && \text{on } \Gamma.
		\end{aligned}
		\right.
	\end{equation}
	Moreover, there holds that
	\begin{equation}
		\label{eq:cost-func-2nd-directional-der-limit}
		\frac{1}{\rho^2} \left[J(\bm{u} + \rho \bm{h}) - J(\bm{u} ) - \rho J'(\bm{u})\bm{h} - \frac{\rho^2}{2} J''(\bm{u})\bm{h}^2   \right] \to 0 \quad \text{as } \rho \to 0^+,
	\end{equation}
	where $J''(\bm{u})\bm{h}^2$ is defined by
	\begin{equation} 
		\label{eq:cost-func-2nd-directional-der-def}
		J''(\bm{u})\bm{h}^2 
		  := \int_{\Omega} \left[1 - e^{S(\bm{u})} \varphi_{\bm{u}} \right] z_{\bm{u}; \bm{h}}^2 dx 	 + \nu \sum_{i=1}^k h_i^2.
	\end{equation}
\end{proposition}
\begin{proof}
	The existence and uniqueness of solutions in $W^{1,s_*}_0(\Omega)$ 
	for some $s_* :=  s_*(r) >2$ to \eqref{eq:adjoint-state0} results from the proof of \cite[Thm.~5.1]{Nhu2025}.
	The directional differentiability of $J$ as well as  the first identity in \eqref{eq:cost-func-1st-directional-der} follows from \cref{thm:directional-diff-control2state} and the chain rule. For the second identity in \eqref{eq:cost-func-1st-directional-der}, testing the equation for $z_{\bm{u}; \bm{h}}$ (see \eqref{eq:1st-directional-deri-S}) and \eqref{eq:adjoint-state0} by $\varphi_{\bm{u}}$ and  $z_{\bm{u}; \bm{h}}$, respectively, yields
	\begin{align*}
		& \int_\Omega \nabla \varphi_{\bm{u}} \cdot \nabla z_{\bm{u}; \bm{h}} + e^{S(\bm{u})} \varphi_{\bm{u}} z_{\bm{u}; \bm{h}} dx = \sum_{i=1}^k \varphi_{\bm{u}}(x_i) h_i 
		\intertext{and}
		& \int_\Omega \nabla \varphi_{\bm{u}} \cdot \nabla z_{\bm{u}; \bm{h}} + e^{S(\bm{u})} \varphi_{\bm{u}} z_{\bm{u}; \bm{h}} dx = \int_{\Omega} (S(\bm{u}) - y_d)z_{\bm{u}; \bm{h}} dx.
	\end{align*}
	The last two equations show the second identity in \eqref{eq:cost-func-1st-directional-der}.

	It remains to verify \eqref{eq:cost-func-2nd-directional-der-limit}. To achieve this, for any $\rho >0$, we set for simplicity 
	\[
		y := S(\bm{u}), \quad y_\rho := S(\bm{u} + \rho \bm{h}), \quad z_\rho := \frac{1}{\rho} (S(\bm{u} + \rho \bm{u}) - S(\bm{u})), \quad z := z_{\bm{u}; \bm{h}} = S'(\bm{u})(\bm{h}). 
	\]
	Obviously, by subtracting the equations for $y_\rho$ and $y$, and then dividing the obtained result via $\rho$, we conclude that   $z_\rho$ fulfills the equation
	\[
		\left\{
		\begin{aligned}
			-\Delta z_\rho  + e^{y} \frac{1}{\rho} (e^{\rho z_\rho} -1) & = \sum_{i=1}^k h_i \delta_{x_i} && \text{in } \Omega \\
			z_\rho &= 0 && \text{on } \Gamma.
		\end{aligned}
		\right.
	\]
	This, along with \eqref{eq:1st-directional-deri-S}, gives
	\[
		\left\{
		\begin{aligned}
			-\Delta (z_\rho - z)  + e^{y} \left[ \frac{1}{\rho} (e^{\rho z_\rho} -1) - z \right]  & =0 && \text{in } \Omega \\
			(z_\rho - z) &= 0 && \text{on } \Gamma,
		\end{aligned}
		\right.
	\]
	or, equivalently,
	\begin{equation}
		\label{eq:zrho-z}
		\left\{
		\begin{aligned}
			-\Delta (z_\rho - z)  + e^{y}(z_\rho - z) & = -e^y  \frac{1}{\rho} (e^{\rho z_\rho} -1 - \rho z_\rho)    && \text{in } \Omega \\
			(z_\rho - z) &= 0 && \text{on } \Gamma.
		\end{aligned}
		\right.
	\end{equation}
	Testing now \eqref{eq:zrho-z} and \eqref{eq:adjoint-state0}, respectively, by $\varphi_{\bm{u}}$ and $(z_\rho - z)$, and thus subtracting the obtained identities, we have
	\begin{equation}
		\label{eq:zrho-z-varphi}
		\int_\Omega (y - y_d)(z_\rho - z)dx = -\frac{1}{\rho} \int_\Omega e^y \varphi_{\bm{u}}  (e^{\rho z_\rho} -1 - \rho z_\rho)dx.
	\end{equation}
	On the other hand, by setting
	\[
		M_\rho := J(\bm{u} + \rho \bm{h}) - J(\bm{u} ) - \rho J'(\bm{u})\bm{h} - \frac{\rho^2}{2} J''(\bm{u})\bm{h}^2
	\]
	and using the identity 
	\[
		\norm{a}_{L^2(\Omega)}^2 - \norm{b}_{L^2(\Omega)}^2 = \norm{a -b}_{L^2(\Omega)}^2 +  2 \int_\Omega (a-b)b dx \quad \text{for all } a, b \in L^2(\Omega)
	\]
	together with the first identity in \eqref{eq:cost-func-1st-directional-der} and the definition in \eqref{eq:cost-func-2nd-directional-der-def}, we derive
	\begin{multline*}
		M_\rho  = \frac{1}{2}\left[\norm{y_\rho - y}_{L^2(\Omega)}^2 + 2 \int_\Omega (y_\rho - y)(y - y_d)dx  \right] \\
		\begin{aligned}
			& \qquad \qquad - \rho \int_\Omega (y - y_d)z dx  - \frac{1}{2} \rho^2\int_{\Omega} \left[1 - e^{y} \varphi_{\bm{u}} \right] z^2 dx  \\
			& = \frac{1}{2} \rho^2 \left[ \norm{ z_\rho}_{L^2(\Omega)}^2  -  \norm{ z}_{L^2(\Omega)}^2 \right] + \rho \int_\Omega (y - y_d)( z_\rho - z ) dx  + \frac{1}{2} \rho^2\int_{\Omega}  e^{y} \varphi_{\bm{u}}z^2 dx.
		\end{aligned}		 
	\end{multline*}
	Combining this with \eqref{eq:zrho-z-varphi} yields
	\begin{multline*}
		M_\rho =  \frac{1}{2} \rho^2 \left[ \norm{ z_\rho}_{L^2(\Omega)}^2  -  \norm{ z}_{L^2(\Omega)}^2 \right] 
		- \int_{\Omega}  e^{y} \varphi_{\bm{u}}\left[  e^{\rho z_\rho} -1 - \rho z_\rho - \frac{1}{2} \rho^2 z^2 \right] dx\\
		\begin{aligned}
			& 	= \frac{1}{2} \rho^2 \left[ \norm{ z_\rho}_{L^2(\Omega)}^2  -  \norm{ z}_{L^2(\Omega)}^2 \right] \\
			& \qquad
			- \int_{\Omega}  e^{y} \varphi_{\bm{u}}\left[  \left(e^{\rho z_\rho} -1 - \rho z_\rho - \frac{1}{2}\rho^2 z_\rho^2\right) + \frac{1}{2}\rho^2 \left(z_\rho^2 - z^2\right) \right] dx.
		\end{aligned}
	\end{multline*}
	From this, \eqref{eq:zrho-z-limit}, and \eqref{eq:expS-2st-Taylor}, we obtain \eqref{eq:cost-func-2nd-directional-der-limit}.
\end{proof}


\begin{definition}
	\label{def:generalized-deri-cost-func}
	For any $\bm{u} = (u_i) \in \mathcal{D}$ with $e^{S(\bm{u})} \in L^r(\Omega)$ for some $r > 1$, we define  \emph{first-order generalized derivative} and  \emph{second-order generalized derivative} of $J$ at $\bm{u}$ as the mappings
	\[
		DJ(\bm{u}) : \ell^1 \to \R \quad \text{and} \quad D^2J(\bm{u}): \ell^1 \times \ell^1 \to \R,
	\]
	respectively, defined via
	\begin{equation}
		\label{eq:generalized-deri-J-1st}
		DJ(\bm{u})\bm{h} := \int_\Omega (S(\bm{u}) - y_d) DS(\bm{u})\bm{h} dx + \nu \sum_{i=1}^\infty u_i h_i 
		= \sum_{i=1}^\infty \varphi_{\bm{u}}(x_i) h_i + \nu \sum_{i=1}^\infty u_i h_i
	\end{equation}
	for all $\bm{h} = (h_i) \in \ell^1$ 
	and
	\begin{equation}
		\label{eq:generalized-deri-J-2nd}
		D^2J(\bm{u})[\bm{h},\bm{k}] := \int_{\Omega} \left[1 - e^{S(\bm{u})} \varphi_{\bm{u}} \right] (DS(\bm{u})\bm{h}) (DS(\bm{u})\bm{k}) dx 	 + \nu \sum_{i=1}^\infty h_i k_i
	\end{equation}
	for all $\bm{h} = (h_i), \bm{k} = (k_i) \in \ell^1$. Here $\varphi_{\bm{u}}$ uniquely satisfies \eqref{eq:adjoint-state0}.
\end{definition}
\begin{proposition}
	\label{prop:generalized-deri-cost-func-limits}
	Let $\bm{u} = (u_i) \in \mathcal{D}$ be such that $e^{S(\bm{u})} \in L^r(\Omega)$ for some $r > 1$. Then, there hold for any $\bm{h} \in \ell^1$ that
	\begin{align}
		& J'(\bm{u})\bm{h}_k \to DJ(\bm{u})\bm{h} \label{eq:1st-deri-J-limit} 
		\intertext{and that}
		& J''(\bm{u})\bm{h}_k^2 \to D^2J(\bm{u})[\bm{h}, \bm{h} ]\label{eq:2nd-deri-J-limit} 
	\end{align}
	as $k \to \infty$.
\end{proposition}
\begin{proof}
	The limits in \eqref{eq:1st-deri-J-limit} and \eqref{eq:2nd-deri-J-limit} follow from \eqref{eq:generalized-deri-J-1st}, \eqref{eq:generalized-deri-J-2nd}, \eqref{eq:cost-func-1st-directional-der}, \eqref{eq:cost-func-2nd-directional-der-def},  \eqref{eq:generalized-dir-S-limit}, and the continuous embeddings
	\[
		S(\bm{u}) \in W^{1,q}_0(\Omega), \quad  \varphi_{\bm{u}} \in W^{1,s_*}(\Omega) \hookrightarrow C(\overline{\Omega}), \quad W^{1,q}_0(\Omega) \hookrightarrow L^{2r'}(\Omega) \hookrightarrow L^2(\Omega)
	\]
	for all $q \in (1,2)$ sufficiently close to $2$, where $s_* >2$ and $r' := \frac{r}{r-1}$.
\end{proof}

\section{Optimal control problem {\eqref{eq:P}}}
\label{sec:optimal-control-prob}

In this section, we derive the fist- and second-order optimality conditions for a local minimizer of \eqref{eq:P}. 
We begin by establishing the existence of minimizers of \eqref{eq:P}.

\subsection{Existence of minimizers}
\label{sec:existence-minimizers}

The set of admissible controls for \eqref{eq:P} is defined as
\begin{equation}
	\label{eq:admissible-set}
	U_{ad} := \left\{ \bm{u} \in \ell^1 \mid \bm{\alpha} \leq \bm{u} \leq \bm{\beta} \right\}.
\end{equation}

The nonemptiness,  convexity, and compactness of $U_{ad}$ are shown in the following.
\begin{lemma}
	\label{lem:admissible-set}
	The feasible set $U_{ad}$ is nonempty, convex, and compact in $\ell^1$.
	Moreover, there exists a constant $\bar r := \bar r(\bm{\beta}) >1$ such that
	\begin{equation}
		\label{eq:expS-belong-Lr-Uad}
		e^{S(\bm{u})} \in L^{\bar r}(\Omega) \quad \text{for all }  \bm{u} \in U_{ad}.
	\end{equation}
\end{lemma}
\begin{proof}
	Clearly, $U_{ad}$ is nonempty, closed, and convex.
	To establish its compactness, we apply the Dunford--Pettis theorem in the context of the $\ell^1$ space and use the Schur property of $\ell^1$; see, e.g., \cite[Thms.~15.4, 15.5]{Voigt2020}, \cite{Megginson1998}, \cite[Thm.~4]{HancheOlsen2010}. 
	Indeed, for any $\bm{u} \in U_{ad}$, there holds
	\begin{equation}
		\label{eq:eta-i-esti}
		|u_i| \leq \max\{|\alpha_i|, |\beta_i| \} \leq |\alpha_i| + |\beta_i| \quad \text{for all } i \geq 1.
	\end{equation}
	Consequently, one has
	\[
		\sum_{i=1}^\infty |u_i| \leq \norm{\alpha}_{\ell^1} + \norm{\beta}_{\ell^1}.
	\]
	Since $\bm{u}$ is arbitrary in $U_{ad}$, $U_{ad}$ is pointwise bounded. Moreover, let $\epsilon >0$ be arbitrary, but fixed.
	Since $\bm{\alpha}, \bm{\beta} \in \ell^1$,  there exists an integer $n_0 >0$ such that
	\[
		\sum_{i=n_0}^\infty |\alpha_i| + |\beta_i| < \epsilon.
	\]
	Combining this with \eqref{eq:eta-i-esti} yields
	\[
		\sup\left\{ \sum_{i=n_0}^\infty |u_i| \middle | \bm{u} \in U_{ad} \right\} \leq \sum_{i=n_0}^\infty |\alpha_i| + |\beta_i| < \epsilon.
	\]
	This shows the equi-integrability of $U_{ad}$. By the Dunford--Pettis theorem, it follows that  $U_{ad}$ is relatively compact in the weak topology. On the other hand, since $\ell^1$ has Schur's property--i.e., weak convergence implies norm convergence--$U_{ad}$	is also relatively compact in norm. Combining this with the closedness of $U_{ad}$, we conclude that $U_{ad}$ 	is compact.

	\medskip
	
	It remains to prove \eqref{eq:expS-belong-Lr-Uad}. To this end, we apply \Cref{lem:expS-esti} to $\bm{u} := \bm{\beta}$ and use the assumption \eqref{eq:beta-rho-ass} to conclude that a constant $\bar r := \bar r(\bm{\beta}) >1$ exists and satisfies $e^{S(\bm{\beta})} \in L^{\bar r}(\Omega)$. Combining this with \eqref{eq:exp-S-belong-Lr} and \eqref{eq:exp-S-belong-Lr-order}, we arrive at \eqref{eq:expS-belong-Lr-Uad}.
\end{proof}

The compactness of the feasible set $U_{ad}$, together with the continuity of the control-to-state operator $S$ established in \cref{lem:control2state-oper-continuity}, ensures the existence of minimizers to \eqref{eq:P}, as shown in the next proposition. 
\begin{proposition}
	\label{prop:existence-minimizer}
	Problem \eqref{eq:P} admits at least one minimizer $\bm{\bar u} \in U_{ad}$. Moreover, any minimizer $\bm{\bar u}$ to \eqref{eq:P} satisfies
	\begin{equation}
		\label{eq:ybar-Lr}
		e^{S(\bm{\bar u})} \in L^{\bar r}(\Omega),
	\end{equation}	
	where $\bar r >1$ is defined as in \Cref{lem:admissible-set}.
\end{proposition}
\begin{proof}
	Since $\bm{\beta} \in U_{ad}$, there holds
	\[
		0 \leq \gamma := \inf \left\{ J(\bm{u}) \middle | \bm{u} \in U_{ad} \right\} \leq J(\bm{\beta}) = \frac{1}{2} \norm{y_{\bm\beta} - y_d}_{L^2(\Omega)}^2 + \nu \sum_{i=1}^\infty \beta_i^2
	\]
	with $y_{\bm\beta} := S(\bm{\beta}) \in W^{1,q}_0(\Omega) \hookrightarrow L^2(\Omega)$ for all $q \in [1,2)$.  There thus exists a sequence $\{ \bm{u}_n \} \subset U_{ad}$ such that
	\begin{equation}
		\label{eq:minimizing-sequence}
		\gamma = \lim \limits_{n \to \infty} J(\bm{u}_n).
	\end{equation}
	Moreover, in view of \cref{lem:admissible-set}, $U_{ad}$ is compact in $\ell^1$. A subsequence, denoted in the same way, of $\{ \bm{u}_n \}$ exists and converges to an element $\bm{\bar u} \in U_{ad}$. From this and \eqref{eq:S-continuity-W1q}, along with the continuous embedding $W^{1,q}_0(\Omega) \hookrightarrow L^2(\Omega)$ for all $q \in [1,2)$, we deduce that
	\[
		J(\bm{u}_n) \to J(\bm{\bar u}) \quad \text{as } n \to \infty,
	\]
	which, in combination with \eqref{eq:minimizing-sequence}, yields $\gamma = J(\bm{\bar u})$. Hence, $\bm{\bar u}$ is a minimizer of \eqref{eq:P}.
	Clearly, any minimizer  of \eqref{eq:P} satisfies \eqref{eq:ybar-Lr} thanks to \eqref{eq:expS-belong-Lr-Uad}.
\end{proof}

\subsection{First- and second-order necessary optimality conditions}
\label{sec:NOCs}
To derive the first-order necessary optimality conditions for \eqref{eq:P}, we first study the $k$-dimensional directions, $k \ge 1$, in the tangent cone to the admissible set $U_{ad}$ at a local minimizer $\bm{\bar u}$ (whose existence is guaranteed by \cref{prop:existence-minimizer}) and then pass to the limit as $k \to \infty$. Similarly, by considering $k$-dimensional \emph{critical} directions and passing to the limit, we obtain the desired second-order necessary optimality conditions for \eqref{eq:P}.

\medskip

We first present a criterion in terms of tangent directions to the admissible set at a point.
\begin{lemma}
	\label{lem:direction-approximations}
	Assume that $\bm{h} = (h_i) \in \ell^1$. Then, 
	$\bm{h} \in \overline{\mathrm{cone}}(U_{ad} - \bm{\bar u})$ if and only if there holds
		\begin{equation} 
			\label{eq:tangent-direction-equivalence}
			h_i 
			\left\{
			\begin{aligned}
				& \geq 0 && \text{if } \bar u_i = \alpha_i,\\
				& \leq 0 && \text{if } \bar u_i = \beta_i
			\end{aligned}
			\right. \quad \text{for all } i \geq 1.
		\end{equation} 
	Moreover, if $\bm{h} \in \ell^1$ fulfills \eqref{eq:tangent-direction-equivalence}, then   $\bm{h}_k  \in \mathrm{cone}(U_{ad} - \bm{\bar u})$.
\end{lemma}
\begin{proof}
	Obviously, any $\bm{h}$ in $\overline{\mathrm{cone}}(U_{ad} - \bm{\bar u})$ satisfies \eqref{eq:tangent-direction-equivalence}. Conversely, let $\bm{h} = (h_i) \in \ell^1$ satisfy \eqref{eq:tangent-direction-equivalence}. 
	To prove $\bm{h} \in \overline{\mathrm{cone}}(U_{ad} - \bm{\bar u})$, it suffices to show that $\bm{h}_k  \in \mathrm{cone}(U_{ad} - \bm{\bar u})$ for all $k \geq 1$, thanks to the fact that $\bm{h}_k \to \bm{h}$ in $\ell^1$ as $k \to \infty$. To this end, we first deduce from \eqref{eq:tangent-direction-equivalence} and the assumption $\beta_i \geq \alpha_i$ for all $1 \leq i \leq k$ that there exists a constant $t_k >0$ fulfilling
	\[
		\alpha_i \leq \bar u_i + t_k h_i \leq \beta_i \quad \text{for all } 1 \leq i \leq k.
	\]
	Choosing now $\bm{u} := (u_i)$ by
	\[
			u_i :=
		\left\{
		\begin{aligned}
			& \bar u_i + t_k h_i  && \text{if } 1 \leq i \leq k,\\
			& \bar u_i && \text{if } i \geq k+1
		\end{aligned}
		\right.
	\]
	yields $\bm{u} \in U_{ad}$ and $\bm{u} = \bm{\bar u} + t_k\bm{h}_k$. This shows $\bm{h}_k \in \mathrm{cone}(U_{ad} - \bm{\bar u})$. Consequently, there holds $\bm{h} \in \overline{\mathrm{cone}}(U_{ad} - \bm{\bar u})$.
\end{proof}

\begin{theorem}[First-order necessary optimality conditions]
	\label{thm:1st-NOCs}
	Assume that $\bm{\bar u}$ is a local minimizer to \eqref{eq:P}. Then, there exists an adjoint state $\bar{\varphi} \in W^{1,s_*}_0(\Omega)$ for some $s_*>2$, 
	which, along with ${\bm{\bar u}}$ and $\bar y := S(\bar{\bm{u}})$, fulfills the system
	\begin{subequations}
		\label{eq:1st-NOCs}
		\begin{align}
			& \left\{
			\begin{aligned}
				-\Delta \bar y + (e^{\bar y} - 1) &= f_0(x) + \sum_{i =1}^\infty \bar u_i \delta_{x_i} && \text{in } \Omega, \\
				\bar y &= 0 && \text{on } \Gamma,
			\end{aligned}
			\right.  \label{eq:1st-NOC-state} \\
			& \left\{
			\begin{aligned}
				-\Delta \bar \varphi + e^{\bar y} \bar\varphi &= \bar y - y_d && \text{in } \Omega, \\
				\bar \varphi &= 0 && \text{on } \Gamma,
			\end{aligned}
			\right. \label{eq:1st-NOC-adjoint-state} 
			\intertext{and}
			& \left\{
			\begin{aligned}
				\bar\varphi(x_i) + \nu \bar u_i & \geq 0 && \text{if } \bar u_i = \alpha_i < \beta_i, \\
				\bar\varphi(x_i) + \nu \bar u_i & = 0 && \text{if } \bar u_i \in (\alpha_i, \beta_i), \\
				\bar\varphi(x_i) + \nu \bar u_i & \leq 0 && \text{if } \bar u_i = \beta_i > \alpha_i
			\end{aligned}
			\right.
			\quad \text{for } i \geq 1. \label{eq:1st-NOC-variational} 
		\end{align}
	\end{subequations}
\end{theorem}
\begin{proof}
	Obviously, $\bm{\bar u}$ and $\bar y$ fulfill \eqref{eq:1st-NOC-state}.
	For any $\bm{u} \in U_{ad}$, we set $\bm{h} := \bm{u} - \bm{\bar u}$ and thus $\bm{h} \in U_{ad} - \bm{\bar u}$.
	Fixing now $k$ yields
	\[
		\bm{u}_\rho := \bm{\bar u} + \rho \bm{h}_k \in U_{ad} \quad \text{for all } 0 < \rho < \rho_0 
	\]
	for some constant $\rho_0$ depending only on $k$. 	From this and the optimality of $\bm{\bar u}$, there holds
	\[
		\frac{1}{\rho} \left[ J(\bm{u}_\rho) - J(\bm{\bar u})\right] \geq 0 \quad \text{for all } 0 < \rho < \rho_0.
	\]
	In view of the $k$-dimensional directional differentiability of $J$; see \cref{prop:cost-func-1st-2nd-directional-der}, there holds
	\[
		J'(\bm{\bar u}) \bm{h}_k \geq 0,
	\]
	or, equivalently, 
	\[
		\sum_{i=1}^k \bar\varphi(x_i) (u_i - \bar u_i) + \nu \sum_{i=1}^k \bar u_i (u_i - \bar u_i) \geq 0,
	\]
	where ${\bar{\varphi}}$ satisfies \eqref{eq:1st-NOC-adjoint-state} according to \cref{prop:cost-func-1st-2nd-directional-der}.
	By taking $k \to \infty$, we obtain
	\begin{equation*}
		\sum_{i=1}^\infty [\bar\varphi(x_i) + \nu \bar u_i] (u_i - \bar u_i) \geq 0 \quad \text{for all } \bm{u} = (u_i) \in U_{ad}.
	\end{equation*}
	This, together with a straightforward argument, gives \eqref{eq:1st-NOC-variational}.
\end{proof}

From now on, let $\bm{\bar u}$,  ${\bar y}$, and $\bar\varphi$ satisfy \eqref{eq:1st-NOCs}. We define  $ \bm{\bar d} = (\bar d_i) \in \ell^1$ by
\begin{equation}
	\label{eq:cost-func-deri-at-minimizer}
	\bar d_i := \bar{\varphi}(x_i) + \nu \bar u_i, \quad i \geq 1. 
\end{equation}
In view of \cref{def:generalized-deri-cost-func}, there holds
\begin{equation}
	\label{eq:J-der-at-minimizer}
	DJ(\bm{\bar u}) \bm{h} = \sum_{i =1}^\infty \bar{d}_i h_i, \quad \text{for all } \bm{h} = (h_i) \in \ell^1.
\end{equation}
As a result of the combination of \eqref{eq:1st-NOC-variational} in \cref{thm:1st-NOCs} with \eqref{eq:tangent-direction-equivalence}, one has
\begin{equation}
	\label{eq:1st-NOC-generalized}
	DJ(\bm{\bar u}) \bm{h} \geq 0 \quad \text{for all } \bm{h} \in \overline{\text{cone}}(U_{ad} - \bm{\bar u}).
\end{equation}
This result extends the classical theory of first-order optimality conditions in smooth optimization.

\medskip 
In order to establish the second-order necessary optimality conditions,
we now define the \emph{critial} cone to \eqref{eq:P} at $\bm{\bar u}$ in the usual way
\begin{equation}
	\label{eq:critical-cone}
	C_{\bm{\bar u}} := \left\{ \bm{h} = (h_i)\in \ell^1 \middle | \bm{h } \, \text{satisfies \eqref{eq:tangent-direction-equivalence} and }  \sum_{i =1}^\infty \bar{d}_i h_i = 0
	\right\}.
\end{equation}

\begin{lemma}
	\label{lem:critical-direction}
	Assume that \eqref{eq:1st-NOC-variational} holds. Let $\bm{h} \in \ell^1$ satisfy \eqref{eq:tangent-direction-equivalence}. Then, there holds
	\begin{equation}
		\label{eq:critical-direction-equiv}
		\bm{h} \in C_{\bm{\bar u}} \quad \Leftrightarrow \quad    h_i = 0 \, \text{if } \bar d_i \neq 0 \, \text{for all } i \geq 1.
	\end{equation}
\end{lemma}
\begin{proof}
	The implication ``$\Leftarrow$'' is straightforward. Conversely, assume that $\bm{h} = (h_i) \in C_{\bm{\bar u}}$. 
	By definition, there holds 
	\[
		\sum_{i=1}^\infty h_i \bar d_i =0.
	\]
	Owing to \eqref{eq:tangent-direction-equivalence}, \eqref{eq:1st-NOC-variational}, and the definition of $\bm{\bar d}$ in \eqref{eq:cost-func-deri-at-minimizer}, we have $h_i \bar d_i \geq 0$ for all $i \geq 1$. There thus holds that 
	\[
		h_i = 0 \quad \text{if} \quad \bar d_i \neq  0 \quad \text{for all } i \geq 1.
	\]
	This proves the converse implication ``$\Rightarrow$''.
\end{proof}

\begin{theorem}
	\label{thm:2ndNOCs}
	Assume that $\bm{\bar u}$ is a local minimizer to \eqref{eq:P}. Then, there exists an adjoint state $\bar\varphi$ in $W^{1,s_*}_0(\Omega)$ for some $s_*>2$, satisfying \eqref{eq:1st-NOCs} and the following second-order necessary optimality condition
	\begin{equation}
		\label{eq:2nd-NOCs}
		D^2J(\bm{\bar u})[\bm{h}, \bm{h}] \geq 0 \quad \text{for all} \quad \bm{h} \in C_{\bm{\bar u}},
	\end{equation}
	where
	\[
		D^2J(\bm{\bar u})[\bm{h}, \bm{h}]  
		= \int_{\Omega} \left[1 - e^{S(\bm{\bar u})} \bar \varphi \right] (DS(\bm{\bar u}) \bm{h})^2 dx 	 + \nu \sum_{i=1}^\infty h_i^2.
	\]
\end{theorem}
\begin{proof}
	On account of the limit \eqref{eq:2nd-deri-J-limit}, it suffices to prove that
	\begin{equation}
		\label{eq:2nd-NOCs-finite-dim}
		J''(\bm{\bar u})\bm{h}^2_k \geq 0 
	\end{equation}
	for all $k \geq 1$ and $\bm{h} \in C_{\bm{\bar u}}$. To this end, Let $\bm{h} \in C_{\bm{\bar u}}$ be arbitrary but fixed. 
	By \eqref{eq:J-der-at-minimizer} and \cref{lem:direction-approximations,lem:critical-direction}, there hold
	\begin{equation}
		\label{eq:hk-direction}
		J'(\bm{\bar u})\bm{h}_k =0 \quad \text{and} \quad \bm{h}_k \in \text{cone}(U_{ad} - \bm{\bar u}) \quad \text{for all } k \geq 1.
	\end{equation}
	From the identity in \eqref{eq:hk-direction} and \eqref{eq:cost-func-2nd-directional-der-limit}, we deduce for each $k\geq 1$ that
	\begin{equation*}
		J''(\bm{\bar u})\bm{h}^2_k =2 \lim\limits_{\rho \to 0^+} \frac{J(\bm{\bar u} + \rho \bm{h}_k) - J(\bm{\bar u})}{\rho^2}.
	\end{equation*}
	Combining this with the optimality of $\bm{\bar u}$ and the second relation in \eqref{eq:hk-direction} gives \eqref{eq:2nd-NOCs-finite-dim}.
\end{proof}

\section{Conclusions}
\label{sec:conclusions}
We have investigated the finite-dimensional directional differentiability of the control-to-state operator for the optimal control problem governed by  exponential semilinear elliptic equations with discrete measures as  sources. 
We  then introduced the notion of generalized differentiability, defined as the limit of the directional derivatives 
of the control-to-state mapping along finite-dimensional truncated directions. 
Building on this concept, we analyzed the first- and second-order generalized  differentiability of the objective functional
and employed these results to derive first- and second-order necessary optimality conditions for local minimizers.
The derivation of second-order sufficient optimality conditions is deferred to future work, as it requires additional technical arguments.
Finally, a practically relevant direction for further research is the establishment of error estimates for approximations of the optimal control problem, as achieved in settings where the nonlinearity in the state equation satisfies the polynomial growth assumption \cite{Otarola2024,AllendesOtarolaRankinSalgado2018}.

\appendix
\section{Some estimates for the exponential function}
\begin{lemma}
	\label{lem:exponential-esti}
	Let $a \in \R$ and $t_0 >0$ be arbitrary. Then, there hold
	\begin{align}
		& 0 \leq \frac{e^{at} - 1 - at}{t} \leq \frac{e^{at_0} - 1 - at_0}{t_0} \label{eq:expo-est-1st}
		\intertext{and}
		& \left| \frac{e^{at} - 1 - at - \frac{1}{2} a^2t^2}{\frac{1}{2} t^2 }\right| \leq \left| \frac{e^{at_0} - 1 - at_0 - \frac{1}{2} a^2t_0^2}{\frac{1}{2} t^2_0 }\right| \label{eq:expo-est-2nd}
	\end{align}
	for all $0 < t < t_0$.
\end{lemma}
\begin{proof}
	Setting the function $\psi(t) := \frac{e^{at} - 1 - at}{t}$, $t \in (0,t_0)$, 
	yields
	\[
		\psi'(t) = \frac{(at-1)e^{at}+1}{t^2} =: \frac{g(t)}{t^2}.
	\]
	Moreover, by a simple computation, we have $g'(t) = a^2te^{at} \geq 0$ for all $0 < t <t_0$ and $g(0) =0$. We thus obtain $\psi'(t) \geq 0$ for all $0 < t <t_0$. Hence, the function $\psi$
	is nondecreasing and continuous on $(0,t_0)$. From this, we obtain \eqref{eq:expo-est-1st}.
	For the estimate \eqref{eq:expo-est-2nd}, we first set
	\[
		\phi(t) := \frac{e^{at} - 1 - at - \frac{1}{2} a^2t^2}{\frac{1}{2} t^2 }, \quad 0 < t < t_0.
	\]
	Since $\phi$ is continuous on $(0,t_0)$ and $\phi(t) \to 0$ as $t \to 0^{+}$, it suffices to show that $\phi$ is monotone on $(0,t_0)$. To this end, by a straightforward computation, we get
	\[
		\phi'(t) = \frac{e^{at}(at-2) +2  + at }{\frac{1}{2} t^3 } =: \frac{h(t)}{\frac{1}{2} t^3}.
	\]
	By a simple computation, one has
	\[
		 h'(t) = ae^{at}(at-1) + a \quad \text{and} \quad h''(t) = a^3 e^{at}t.
	\]
	We now consider two cases:
	
	\emph{Case I: $a \geq 0$.} In this case, there holds $h''(t) \geq 0$ for all $t \in (0,t_0)$. Consequently, we have $h'(t) \geq h'(0)=0$ for all $t \in (0,t_0)$. Thus, $h$ is nondecreasing on $(0,t_0)$ and $h(t) \geq h(0) =0$ for all $t \in (0,t_0)$. Therefore, $\phi$ is nondecreasing on $(0,t_0)$.
	
	\emph{Case II: $a < 0$.} Similarly, $\phi$ is nonincreasing on $(0,t_0)$.
\end{proof}

\section{Exponential estimates for Poisson's equation with regularized point sources}
\begin{lemma}
	\label{lem:Poisson-regularized-point-source}
	Let $x_0 \in \R^2$ and $0 < \rho_0 <R$ be such that $B_{\R^2}(x_0, \rho_0) \subset B_{\R^2}(0,R) =: B_R$.
	Assume that $y_0$ is the unique solution to
	\begin{equation}
		\label{eq:Poisson-reg-point-source}
		\left\{
		\begin{aligned}
			-\Delta  y_0 & = \phi_\epsilon(\cdot -x_0) && \text{in } B_R, \\
			 y_0 &= 0 && \text{on } \partial B_R,
		\end{aligned}
		\right.
	\end{equation} 
	where $\phi_\epsilon$, $\epsilon \in (0, \rho_0)$, is a mollifier satisfying
	\[
	0 \leq \phi_\epsilon \in C^\infty_c(\R^2), \quad \supp \phi_\epsilon \subset \overline{B_{\R^2}(0, \epsilon)}, \quad \int_{\R^2} \phi_\epsilon dx =1.
	\]
	Then there hold for all $m \in (0, 4\pi)$ that
	\begin{equation}
		\label{eq:y0-esti-pointwise-outside}
		\exp[m {y}_0(x)] \leq  \left( \frac{2R}{\rho_0 -\epsilon} \right)^{ \frac{m}{2 \pi}}
	\end{equation}
	for a.e. $x \in B_R \backslash \overline{  B_{\R^2}(x_0, \rho_0)}$
	and
	\begin{equation}
		\label{eq:y0-esti-inball}
		\int_{ B_{\R^2}(x_0, \rho_0)} \exp[m{y}_0(x)]  dx 
		\leq \frac{2 \pi (\epsilon + \rho_0)^2}{2 - \frac{m}{2 \pi}}  \left( \frac{2R}{\epsilon + \rho_0} \right)^{\frac{m}{2 \pi} }.
	\end{equation}
\end{lemma}
\begin{proof}
	Let $G$ denote the Green function of the Laplace operator $-\Delta$ in the unit ball in $\R^2$. Green's representation formula; see, e.g. \cite[Cor.~2.2]{VeronVivier2000}, \cite[Thm.~12, Chap.~2]{Evans2010}, and \cite{ChipotQuittner2004}, yields
	\begin{equation*}
		{y}_0(x) = \int_{B_R} G\left(\frac{x}{R}, \frac{x'}{R}\right)\phi_\epsilon(x' - x_0)dx' =  \int_{B_{\R^2}(x_0, \epsilon)} G\left(\frac{x}{R}, \frac{x'}{R}\right)\phi_\epsilon(x' - x_0)dx' 
	\end{equation*}
	for a.e. $x \in B_R$. 	
	By  \cite[Lem.~A.1]{Nhu2025}, it follows that
	\begin{equation}
		\label{eq:y0-esti-pointwise}
		0 \leq {y}_0(x) \leq \frac{1}{2 \pi} \int_{B_{\R^2}(x_0, \epsilon)} \ln \left( \frac{2R}{|x-x'|} \right)  \phi_\epsilon(x' - x_0)dx' 
	\end{equation}
	for a.e. $x \in B_R.$
	
	We now consider the following cases:
	
	\noindent$\bullet$
	For a.e.  $x \in B_R \backslash \overline{ B_{\R^2}(x_0, \rho_0)}$, there holds $|x - x'| \geq \rho_0-\epsilon $ for all $x' \in B_{\R^2}(x_0, \epsilon)$. We then have from \eqref{eq:y0-esti-pointwise} that
	\[
	0 \leq {y}_0(x) \leq \frac{1}{2 \pi} \int_{B_{\R^2}(x_0, \epsilon)} \ln \left(\frac{2R}{\rho_0 -\epsilon}\right) \phi_\epsilon(x' - x_0)dx' = \frac{1}{2 \pi}\ln \frac{2R}{\rho_0 -\epsilon}.
	\]
	This gives	 \eqref{eq:y0-esti-pointwise-outside}.
	
	\noindent$\bullet$
	For a.e. $x \in \overline{  B_{\R^2}(x_0, \rho_0)}$, thanks to \eqref{eq:y0-esti-pointwise} and the fact that $\supp \phi_\epsilon \subset \overline{B_{\R^2}(0, \epsilon)}$, we obtain
	\begin{equation*}
		\exp[m{y}_0(x)]
		\leq \exp \left[  \frac{m}{2 \pi} \int_{B_{\R^2}(x_0, \epsilon)} \ln \left( \frac{2R}{|x-x'|} \right)  \phi_\epsilon(x' - x_0)dx' \right].
	\end{equation*}
	Jensen's inequality for integrals; see, .e.g. \cite[Thm.~B.50]{Leoni2017}, and the identity
	\[
		\int_{B_{\R^2}(x_0, \epsilon)}\phi_\epsilon(x' - x_0)dx' =1
	\] 
	then imply that
	\begin{multline*}
		\exp[m{y}_0(x)] \leq 
		\int_{B_{\R^2}(x_0, \epsilon)} \exp \left [  \frac{m}{2 \pi} \ln \left( \frac{2R}{|x-x'|} \right) \right]   \phi_\epsilon(x' - x_0) dx'\\
		\begin{aligned}
			& = 
			\int_{B_{\R^2}(x_0, \epsilon)} \left ( \frac{2R}{|x-x'|} \right )^{\frac{m}{2 \pi }}   \phi_\epsilon(x' - x_0) dx'.
		\end{aligned}
	\end{multline*}
	By integrating over ${ B_{\R^2}(x_0, \rho_0)}$, we then have
	\begin{equation*}
		\int_{ B_{\R^2}(x_0, \rho_0)} \exp[m {y}_0(x)]  dx 	\leq \int_{ B_{\R^2}(x_0, \rho_0)} dx \int_{B_{\R^2}(x_0, \epsilon)} \left ( \frac{2R}{|x-x'|} \right )^{\frac{m}{2 \pi }}   \phi_\epsilon(x' - x_0) dx'.
	\end{equation*}
	From this and a method of  integration by substitution, we deduce 
	\begin{multline*}
		\int_{ B_{\R^2}(x_0, \rho_0)} \exp[m {y}_0(x)] dx  \leq \int_{ B_{\R^2}(0,  \rho_0 + \epsilon)} d \eta \int_{B_{\R^2}(0, \epsilon)} \left( \frac{2R}{|\eta|} \right)^{\frac{m}{2 \pi }}  \phi_\epsilon(\psi) d\psi \\\\
		\begin{aligned}	
			& =  \int_{ B_{\R^2}(0,  \rho_0 + \epsilon)} \left( \frac{2R}{|\eta|} \right)^{\frac{m}{2 \pi }} d \eta 
			= 2 \pi (2R)^{\frac{m}{2 \pi }} \times  \frac{1}{2 - {\frac{m}{2 \pi }} }\left(\rho_0 + \epsilon\right)^{2 -{\frac{m}{2 \pi }} }.
		\end{aligned}
	\end{multline*}
	This shows \eqref{eq:y0-esti-inball}.
\end{proof}


\bibliographystyle{unsrt}  

\bibliography{exponentialmeasure_nfao}

\section*{Statements and Declarations}

\subsection*{Funding}
{%
	This research is funded by Phenikaa University under grant number PU2025-4-A-02.}

\subsection*{Disclosure Statement}
{%
	The authors report there are no competing interests to declare.
}%

\subsection*{Author Contributions}
{%
	All authors contributed to the study conception and design. The first draft of the manuscript was written by the first author and all authors commented on previous versions of the manuscript. All authors read and approved the final manuscript.
}

\subsection*{Data Availability}
{%
	All data generated or analysed during this study are included in this published article [and its supplementary information files].
}

\end{document}